\documentclass[twoside,11pt]{article}

 \usepackage[abbrvbib, preprint]{jmlr2e}

\usepackage{jmlr2e}

\usepackage{amssymb,latexsym,amsmath}
\usepackage{epsfig}
\usepackage{caption}
 \usepackage{float}
\usepackage{amssymb,nicefrac,upgreek,mathtools,verbatim}

\usepackage{amssymb,nicefrac,upgreek,mathtools,verbatim}

\usepackage[mathscr]{eucal}

\usepackage[capitalize,nameinlink]{cleveref}

 \usepackage{amsmath}
\usepackage{amsfonts}
\usepackage{array}
\usepackage{dsfont}
\usepackage{hyperref}
\usepackage{amssymb}

 \usepackage{algorithm}
\usepackage{algorithmic}
\usepackage{pifont}

\usepackage{makeidx}
\usepackage{graphicx}
\graphicspath{ {./images/} }

\allowdisplaybreaks

 \usepackage{graphicx,fancybox,latexsym,epsfig}
\usepackage{fancyhdr,amsmath,times,amsxtra,amssymb}
\usepackage{color}

\newcommand{\sE}{{\mathscr{E}}}

\newcommand{\Uadm}{\mathfrak U}
\newcommand{\Act}{\mathbb{U}}
\newcommand{\Usm}{\mathfrak U_{\mathsf{sm}}}

\newcommand{\Um}{\mathfrak U_{\mathsf{m}}}
\newcommand{\Udsm}{\mathfrak U_{\mathsf{dsm}}}

\newcommand{\pV}{\mathrm{V}} 
\newcommand{\pv}{\mathrm{v}} 

\newcommand{\fB}{{\mathfrak{B}}}  
\newcommand{\cB}{{\mathcal{B}}}  
\newcommand{\sB}{{\mathscr{B}}}  
\newcommand{\cC}{{\mathcal{C}}}   
 
\newcommand{\sF}{{\mathfrak{F}}}   
\newcommand{\cJ}{{\mathcal{J}}}  
  
\newcommand{\sK}{{\mathscr{K}}}  
\newcommand{\sL}{{\mathscr{L}}}  %
\newcommand{\Lp}{{L}}            

\newcommand{\cM}{{\mathcal{M}}} 
\newcommand{\fM}{{\mathfrak{M}}} 
\newcommand{\cP}{{\mathcal{P}}}  

\newcommand{\Lyap}{{\mathcal{V}}}  

\newcommand{\RR}{\mathds{R}}
\newcommand{\NN}{\mathds{N}}
\newcommand{\ZZ}{\mathds{Z}}

\newcommand{\Rd}{{\mathds{R}^{d}}}
\DeclareMathOperator{\Exp}{\mathbb{E}}

\newcommand{\cD}{\mathcal{D}} 
\newcommand{\sD}{{\mathscr{D}}}   
 
\newcommand{\Sob}{{\mathscr W}}    
\newcommand{\Sobl}{{\mathscr W}_{\text{loc}}} 

\newcommand{\df}{:=}
\newcommand{\transp}{^{\mathsf{T}}}

\DeclareMathOperator*{\trace}{Tr}

\newcommand{\sorder}{{\mathfrak{o}}}
\newcommand{\grad}{\nabla}
\newcommand{\uuptau}{{\Breve\uptau}}

\newcommand{\abs}[1]{\lvert#1\rvert}
\newcommand{\norm}[1]{\lVert#1\rVert}

\allowdisplaybreaks

\ShortHeadings{Reinforcement Learning for Control of Diffusion Processes}{Bayraktar, Kara, Pradhan and Y\"uksel}
\firstpageno{1}

\begin{document}
\sloppy
\title{Reinforcement Learning for Discounted and Ergodic Control of Diffusion Processes \thanks{
This research was supported in part by
the Natural Sciences and Engineering Research Council (NSERC) of Canada and the National Science Foundation of the United States.}
}

\author{Erhan Bayraktar, Ali D. Kara, Somnath Pradhan and Serdar Y\"uksel
\thanks{ Erhan Bayraktar is with the Department of Mathematics, University of Michigan, Ann Arbor, MI, USA, Email: erhan@umich.edu.\\
Ali D. Kara is with the Department of Mathematics, Florida State University, Tallahassee, FL, USA, Email: akara@fsu.edu.\\
Somnath Pradhan is with Department of Mathematics, Indian Institute of Science Education and Research Bhopal, Bhopal, MP, India, Email: somnath@iiserb.ac.in\\
S. Y\"uksel is with the Department of Mathematics and Statistics,
     Queen's University, Kingston, ON, Canada,
     Email: yuksel@queensu.ca}
     }

\editor{}

\maketitle

\begin{abstract}
This paper develops a quantized Q-learning algorithm for the optimal control of controlled diffusion processes on $\mathbb{R}^d$ under both discounted and ergodic (average) cost criteria. We first establish near-optimality of finite-state MDP approximations to discrete-time discretizations of the diffusion, then introduce a quantized Q-learning scheme and prove its almost-sure convergence to near-optimal policies for the finite MDP. These policies, when interpolated to continuous time, are shown to be near-optimal for the original diffusion model under discounted costs and -via a vanishing-discount argument- also under ergodic costs for sufficiently small discount factors. The analysis applies under mild conditions (Lipschitz dynamics, non-degeneracy, bounded continuous costs, and Lyapunov stability for ergodic case) without requiring prior knowledge of the system dynamics or restrictions on control policies (beyond admissibility). Our results complement recent work on continuous-time reinforcement learning for diffusions by providing explicit near-optimality rates and extending rigorous guarantees both for discounted cost and ergodic cost criteria for diffusions with unbounded state space.
\end{abstract}



\section{Introduction}

In this paper, we present rigorous reinforcement learning results for controlled diffusions. Notably, we present a Q-learning algorithm and establish its convergence to near optimality for controlled diffusions under discounted and ergodic cost criteria where the state space is $\mathbb{R}^n$ for some $n \in \mathbb{N}$ (and thus, boundedness of the state space is not assumed). Towards this goal, we show that a Q-learning algorithm with piecewise constant control policies applied at discrete-time instances leads to a fixed point which then leads to a near optimal control policy for each of the criteria under mild conditions. These conditions are mild in the sense that they involve only the conditions which are needed for the existence of optimal solutions.


We start by defining the dynamics of the control problem.
The objective is to study a controlled diffusion process, $X(\cdot)$, given by the following stochastic differential equation
\begin{align}\label{diff}
X(t)=X_0+\int_0^{t}b(X(s),U(s))ds+\int_0^t\upsigma(X(s))dW(s)
\end{align}
for $t \geq 0$, where $X(t)\in\mathds{R}^{d}$. $W(\cdot)$ is the driving noise which is assumed to be a Wiener process, and  $U(\cdot)\in\mathds{U}$ is the control process with measurable paths. Here $\mathds{U}$ is a convex compact subset of $\Rd$\,. We assume that the control $U(\cdot)$ is {\it non-anticipative} \cite{ABG-book} such that for $0\leq s <t$ the noise increments, $W(t)-W(s)$, are independent of $W(r),U(r)$ for $r\leq s$. Such a control policy is called an admissible control policy. Let $\Uadm$ be the space of all admissible policies\,. 

 
 We first review that an optimal policy designed for the time discretized model via piecewise constant control policies leading to
  \begin{align*}
 X_{(k+1)h} =& X_{kh} + \int_{kh}^{(k+1)h} b(X_{s},U_{kh})ds + \int_{kh}^{(k+1)h} \upsigma(X_{s}) dW_s
 \end{align*} is near optimal for discounted cost criteria (see Theorem \ref{TD1.3}) as well as for average cost criteria (see Theorem \ref{TNearOpt1A}).

We then show that near optimal policies for the above, obtained via a properly design Quantized Q-Learning algorithm, are near optimal for the controlled diffusion problems: To this end, as our primary contribution, we introduce a  Q-learning algorithm with guaranteed convergence and near optimality properties, in the sense that policies learned via Q-learning (its quantized version as in \cite{KSYContQLearning}) are near optimal for the controlled diffusion model under discounted and average cost criteria. 

There have been closely related several recent studies which complement our analysis in several directions:

A closely related sequence of contributions \cite{wang2020reinforcement,jia2023q,jia2025erratum} consider regularized cost minimizations and the associated continuous-time optimality equations, for models which also allow for control dependence in the diffusion term. Via the HJB equation involving a regularized cost, \cite{wang2020reinforcement} develops a policy improvement method with rigorous guarantees. Along a similar entropy regularized formulation, instead of the discrete-time $Q$ function iterations, via an HJB optimality approximation,  \cite{jia2023q,jia2025erratum} introduce a density function $q_t$ with respect to time deviations of $Q$ updates to compute value functions. The updates are obtained via stochastic approximation equations, though with distinct characteristics when compared with standard discrete-time Q updates. For this analysis, the control policy is apriori restricted to be Lipschitz, and sufficient implicit conditions on optimality under convergence results are presented (such as a martingale condition) while the convergence of the numerical algorithm itself is not formally established.


 \cite{jin2025adaptive} presents an analysis based on empirical value iteration via an adaptive quantization of the state space towards a rigorous regret based sample complexity bounds, where a discrete-time sampled Euler-Maruyama model is considered.

In this paper, we consider controlled diffusion models in the entire unbounded Euclidean state space. Additionally, critically we also consider average cost criterion and also obtain rates of convergence in terms of error bounds when compared with the optimal solution for the original controlled diffusion problem. We do not make apriori restrictions on the assumed control policies. Our analysis allows one to take advantage of the relatively more mature theory of discrete-time stochastic control, as there is an established theory on existence, approximation, and learning theory for optimal discrete-time stochastic control \cite{KSYContQLearning}. In this context, a close study is \cite{bayraktar2022approximate} whose analysis was confined to compact state spaces with discounted cost. 

Our work also utilizes discrete-time approximation results for controlled diffusions. We refer to the following approximation results involving controlled diffusions on finite horizon or discounted cost criteria see, e.g., \cite{BR-02}, \cite{BJ-06}, \cite{KD92}, \cite{KH77}, \cite{KH01}, \cite{KN98A}, \cite{KN2000A}, though the ergodic control and control up to an exit time criteria have also been studied \cite{KD92,kushner2014partial}. For average cost criteria, the analysis has often been restricted to reflected diffusion processes in a smooth bounded domain; we refer the reader to \cite{kushner1990numerical,kushner2001numerical, kushner2012weak,fleming2006controlled}. In an alternative program, building on finite difference approximations for HJB equations utilizing their regularity properties, Krylov \cite{KN98A}, \cite{KN2000A} established the convergence rate for such approximation techniques, where finite difference approximations are studied to arrive at stability results. In particular, several estimates for error bounds on finite-difference approximation schemes in the problem of finding viscosity or probabilistic solutions to degenerate Bellman equations have been established. Also, for controlled non-degenerate diffusion processes, it is shown in \cite{KN99AA} that using policies which are constant on intervals of length $h^2$, one can approximate the value function with errors of order $h^{\frac{1}{3}}$. In \cite{BR-02}, \cite{BJ-06} Barles et. al. improved the error bounds obtained in \cite{KN98A}, \cite{KN2000A}, \cite{KN99AA}; along this note a recent further analysis for more general stochastic control models (including degenerate models) is presented in \cite{jakobsen2019improved} where both stochastic analysis and PDE methods are utilized. As we discuss later in the paper, we build also on the approximation results presented in \cite{pradhan2025discrete}.

Thus, with regard to our rigorous reinforcement learning results for controlled diffusions, the primary contributions of this paper can be summarized as follows: 
\begin{itemize}
\item We consider both discounted and average cost criteria: (i) The latter (average cost criterion) has not been studied to our knowledge in the literature except for \cite[Theorem 12]{jia2023q} in the implicit martingale convergence conditions reviewed above, and (ii) while learning theoretic results are present for the former (discounted cost criterion), a reinforcement learning study involving a discrete (in time and space) approximation with rigorous convergence and error bounds for a diffusion in the whole Euclidean space is novel to our knowledge. 
\item We establish near optimality of the obtained control policies from the reinforcement learning algorithm over all admissible control policies (for the original diffusion model). 
\item We also obtain explicit rates of convergence in terms of the discrete-time approximations.
\item Please see Section \ref{mainResCont} where a technical summary of our main results is presented.

\end{itemize}


\subsection{Problem Setup, Policies, and Cost Criteria}

Let $\Act$ be a convex compact subset of the Euclidean space $\RR^N$ the state space $\mathbb{X} = \RR^d$ and $\pV=\cP(\Act)$ be the space of probability measures on  $\Act$ with topology of weak convergence. Let $$b : \Rd \times \Act \to  \Rd, $$ $$ \upsigma : \Rd \to \RR^{d \times d},\, \upsigma = [\sigma_{ij}(\cdot)]_{1\leq i,j\leq d},$$ be given functions. We consider a stochastic optimal control problem whose state is evolving according to a controlled diffusion process given by the solution of the following stochastic differential equation (SDE)
\begin{equation}\label{E1.1}
d X_t \,=\, b(X_t,U_t) d t + \upsigma(X_t) d W_t\,,
\quad X_0=x\in\Rd.
\end{equation}
Where 
\begin{itemize}
\item
$W$ is a $d$-dimensional standard Wiener process, defined on a complete probability space $(\Omega, \sF, \mathbb{P})$.
\item 
 We extend the drift term $b : \Rd \times \pV \to  \Rd$ as follows:
\begin{equation*}
b (x,\mathrm{v}) = \int_{\Act}b (x,\zeta)\mathrm{v}(d \zeta), 
\end{equation*}
for $\mathrm{v}\in\pV$.
\item
$U$ is a $\pV$ valued process satisfying the following non-anticipativity condition: for $s<t\,,$ $W_t - W_s$ is independent of
\begin{align*}
\sF_s := &\,\,\mbox{the completion of}\,\,\, \sigma(X_0, U_r, W_r : r\leq s)\nonumber\\
&\,\,\,\mbox{relative to} \,\, (\sF, \mathbb{P})\,.
\end{align*}
\end{itemize}
As noted, the process $U$ is called an \emph{admissible} control, and the set of all admissible controls is denoted by $\Uadm$ (see, \cite{BG90}).

To ensure existence and uniqueness of solutions of \cref{E1.1}, we impose the following assumptions on the drift $b$ and the diffusion matrix $\upsigma$\,. 

\begin{itemize}
\item[\hypertarget{A1}{{(A1)}}]
\emph{Lipschitz continuity:\/}
The function
$\upsigma\,=\,\bigl[\upsigma^{ij}\bigr]\colon\RR^{d}\to\RR^{d\times d}$,
$b\colon\Rd\times\Act\to\Rd$ are uniformly bounded and Lipschitz continuous in $x$, $(x, \zeta)$ respectively. In other words, for some constant $C_{0}>0$, we have
\begin{align*}
\abs{b(x,\zeta_1) - b(y, \zeta_2)}^2 &+ \norm{\upsigma(x) - \upsigma(y)}^2 \nonumber\\
&\,\le\, C_{0}\,\left(\abs{x-y}^2 + \abs{\zeta_1- \zeta_2}^2\right)\,,
\end{align*}
for all $x,y\in \Rd$ and $\zeta_1, \zeta_2\in\Act$, where $\norm{\upsigma}\df\sqrt{\trace(\upsigma\upsigma\transp)}$\,.
\medskip
\item[\hypertarget{A2}{{(A2)}}]
\emph{Nondegeneracy:\/}
For each $R>0$, there exists a positive constant $C_{R}$ (depending on $R$) such that
\begin{equation*}
\sum_{i,j=1}^{d} a^{ij}(x)z_{i}z_{j}
\,\ge\,C^{-1}_{R} \abs{z}^{2} \qquad\forall\, x\in \sB_{R}\,,
\end{equation*}
and for all $z=(z_{1},\dotsc,z_{d})\transp\in\RR^{d}$,
where $a\df \frac{1}{2}\upsigma \upsigma\transp$.
\end{itemize}


By a Markov control we mean an admissible control of the form $U_t = v(t,X_t)$ for some Borel measurable function $v:\RR_+\times\Rd\to\pV$. The space of all Markov controls is denoted by $\Um$\,.
If the function $v$ is independent of $t$, then $v$ is called a stationary Markov control. The set of all stationary Markov controls is denoted by $\Usm$. A policy $v\in \Usm$ is said to be a deterministic stationary Markov policy if $v(x) = \delta_{f(x)}$ for some measurable map $f:\Rd\to \Act$. Let $\Udsm$ be space of all deterministic stationary Markov policies\,. From \cite[Section~2.4]{ABG-book}, we have that the set $\Usm$ is metrizable under the following (Borkar) topology, which defines a compact metric under which a sequence $v_n\to v$ in $\Usm$ if and only if
\begin{align*}
&\lim_{n\to\infty}\int_{\Rd}f(x)\int_{\Act}g(x,\zeta)v_{n}(x)(d \zeta)d x \nonumber\\
& = \int_{\Rd}f(x)\int_{\Act}g(x,\zeta)v(x)(d \zeta)d x
\end{align*}
for all $f\in L^1(\Rd)\cap L^2(\Rd)$ and $g\in \cC_b(\Rd\times \Act)$\,. It is well known that under the hypotheses \hyperlink{A1}{{(A1)}}--\hyperlink{A2}{{(A2)}}, for any admissible control \cref{E1.1} has a unique strong solution \cite[Theorem~2.2.4]{ABG-book}, and under any stationary Markov strategy \cref{E1.1} has a unique strong solution which is a strong Feller (therefore strong Markov) process \cite[Theorem~2.2.12]{ABG-book}.

\begin{itemize}
\item[\hypertarget{A3}{{(A3)}}]
The \emph{running cost} function $c\colon\Rd\times\Act \to \RR_+$ is bounded and Lipschitz continuous in $(x, \zeta)$, i.e., for some positive constant $K_c$, we have 
\begin{align*}
\abs{c(x,\zeta_1) - c(y, \zeta_2)}^2 \,\le\, K_c\,\left(\abs{x-y}^2 + \abs{\zeta_1- \zeta_2}^2\right)\,,
\end{align*}
for all $x,y\in \Rd$ and $\zeta_1, \zeta_2\in\Act$\,.
\end{itemize}

We extend $c\colon\Rd\times\pV \to\RR_+$ as follows: for $\pv \in \pV$
\begin{equation*}
c(x,\pv) := \int_{\Act}c(x,\zeta)\pv(d\zeta)\,.
\end{equation*}
\textbf{Discounted Cost Criterion:} For $U \in\Uadm$, the associated \emph{$\alpha$-discounted cost} is given by
\begin{equation}\label{EDiscost}
\cJ_{\alpha}(x, c,U) \,\df\, \Exp_x^{U} \left[\int_0^{\infty} e^{-\alpha s} c(X_s, U_s) d s\right],\quad x\in\Rd\,,
\end{equation} where $\alpha > 0$ is the discount factor
and $X_{\cdot}$ is the solution of \cref{E1.1} corresponding to $U\in\Uadm$ and $\Exp_x^{U}$ is the expectation with respect to the law of the process $X_{\cdot}$ with initial condition $x$. Here the controller tries to minimize \cref{EDiscost} over the set of admissible controls $U\in \Uadm$\,.
A control ${U}^{*}\in \Uadm$ is said to be an optimal control if for all $x\in \Rd$ 
\begin{equation}\label{OPDcost}
\cJ_{\alpha}(x, c,U^*) = \inf_{U\in \Uadm}\cJ_{\alpha}(x, c,U) \,\,\, (=: \,\, J_{\alpha}^*(x,c))\,.
\end{equation}\\ \\
\textbf{Ergodic Cost Criterion:} For $U\in\Uadm$, the associated ergodic cost is defined as 
\begin{equation}\label{ECost1}
\sE_{x}(c, U) = \limsup_{T\to \infty}\frac{1}{T}\Exp_x^{U}\left[\int_0^{T} c(X_s, U_s) d{s}\right]\,.
\end{equation} and the optimal value is defined as
\begin{equation}\label{ECost1Opt}
\sE^*(c) \,\df\, \inf_{x\in\Rd}\inf_{U\in \Uadm}\sE_{x}(c, U)\,.
\end{equation}
Then a control ${U}^*\in \Uadm$ is said to be optimal if we have 
\begin{equation}\label{ECost1Opt1}
\sE_{x}(c, {U}^*) = \sE^*(c)\quad \text{for all} \,\,\, x\in \Rd\,.
\end{equation}
Associated to the controlled diffusion model \cref{E1.1}, we define a family of operators $\sL_{\zeta}$ mapping $\cC^2(\Rd)$ to $\cC(\Rd)$ by
\begin{equation}\label{E-cI}
\sL_{\zeta} f(x) \,\df\, \trace\bigl(a(x)\grad^2 f(x)\bigr) + \,b(x,\zeta)\cdot \grad f(x)\,, 
\end{equation}
for $\zeta\in\Act$, \,\, $f\in \cC^2(\Rd)\cap\cC_b(\Rd)$\,.
For $\pv \in\pV$ we extend $\sL_{\zeta}$ as follows:
\begin{equation}\label{EExI}
\sL_\pv f(x) \,\df\, \int_{\Act} \sL_{\zeta} f(x)\pv(d \zeta)\,.
\end{equation} For $v \in\Usm$, we define
\begin{equation}\label{Efixstra}
\sL_{v} f(x) \,\df\, \trace(a\grad^2 f(x)) + b(x,v(x))\cdot\grad f(x)\,.
\end{equation}

\subsection{Main Results and the Learning Algorithms}\label{mainResCont}

We have four main contributions:

\begin{itemize}
\item[(i)] [Approximation of a Controlled Diffusion by a Finite MDP for Discounted Cost] We show that an optimal  solution of a finite model MDP which approximates a discrete-time approximation of a controlled diffusion (on the whole space) process is near optimal for the diffusion under a discounted cost criterion. This is studied in Section \ref{sectionDiscFM}.
\item[(ii)] [Approximation of a Controlled Diffusion by a Finite MDP for Average Cost] We show that for sufficiently small discounting, an optimal policy (Theorem \ref{TErgoOptNear}) or, more consequentially, a near optimal policy (Theorem \ref{NearOptDisErgo1}) for the discounted cost criterion is near optimal for the average cost criterion. While this result is well-established for finite MDPs due to Blackwell \cite{derman1970finite}\cite{blackwell1962discrete} and only recently studied in Borel MDPs \cite[Theorem 5]{creggZeroDelayNoiseless}, a continuous-time counterpart does not exist to our knowledge. Accordingly, we show that an optimal  solution of a finite model MDP which approximates a discrete-time approximation of a controlled diffusion process (defined on the whole space) is near optimal for the diffusion under the average cost criterion. This is studied in Section \ref{sectionErgFM}
\item[(iii)] Building on (i) above, we show that the policy obtained by the learning algorithm given in Algorithm \ref{disc_alg} leads to a near optimal policy. We show that the algorithm converges and the limit is near optimal.
\item[(iv)] By showing that for sufficiently small discounting, a near optimal policy for the discounted criterion is near optimal for the average cost problem, we show that the algorithm also can be used to arrive at near optimal policies for average cost.
\item[(v)] The proposed algorithm is model-free (that is, with dynamics unknown) and near optimality results impose no a-priori restrictions on control policies except for admissibility.    
\end{itemize}
\subsection*{Notation:}
\begin{itemize}
\item For any set $A\subset\RR^{d}$, by $\uptau(A)$ we denote \emph{first exit time} of the process $\{X_{t}\}$ from the set $A\subset\RR^{d}$, defined by
\begin{equation*}
\uptau(A) \,\df\, \inf\,\{t>0\,\colon X_{t}\not\in A\}\,.
\end{equation*}
\item $\sB_{r}$ denotes the open ball of radius $r$ in $\RR^{d}$, centered at the origin,
\item $\uptau_{r}$, $\uuptau_{r}$ denote the first exist time from $\sB_{r}$, $\sB_{r}^c$ respectively, i.e., $\uptau_{r}\df \uptau(\sB_{r})$, and $\uuptau_{r}\df \uptau(\sB^{c}_{r})$.
\item By $\trace S$ we denote the trace of a square matrix $S$.
\item For any domain $\cD\subset\RR^{d}$, the space $\cC^{k}(\cD)$ ($\cC^{\infty}(\cD)$), $k\ge 0$, denotes the class of all real-valued functions on $\cD$ whose partial derivatives up to and including order $k$ (of any order) exist and are continuous.
\item $\cC_{\mathrm{c}}^k(\cD)$ denotes the subset of $\cC^{k}(\cD)$, $0\le k\le \infty$, consisting of functions that have compact support. This denotes the space of test functions.
\item $\cC_{b}(\Rd)$ denotes the class of bounded continuous functions on $\Rd$\,.
\item $\cC^{k}_{0}(\cD)$, denotes the subspace of $\cC^{k}(\cD)$, $0\le k < \infty$, consisting of functions that vanish in $\cD^c$.
\item $\Lp^{p}(\cD)$, $p\in[1,\infty)$, denotes the Banach space
of (equivalence classes of) measurable functions $f$ satisfying
$\int_{\cD} \abs{f(x)}^{p}\,d{x}<\infty$.
\item $\Sob^{k,p}(\cD)$, $k\ge0$, $p\ge1$ denotes the standard Sobolev space of functions on $\cD$ whose generalized derivatives up to order $k$ are in $\Lp^{p}(\cD)$, equipped with its natural norm (see, \cite{Adams})\,.
\item  If $\mathcal{X}(Q)$ is a space of real-valued functions on $Q$, $\mathcal{X}_{\mathrm{loc}}(Q)$ consists of all functions $f$ such that $f\varphi\in\mathcal{X}(Q)$ for every $\varphi\in\cC_{\mathrm{c}}^{\infty}(Q)$. In a similar fashion, we define $\Sobl^{k, p}(\cD)$.
\end{itemize}

In the following, we present the algorithm. The rest of the paper will focus on the rigorous analysis on convergence and near optimality. 

\begin{algorithm}
\caption{Q-learning Algorithm for Controlled Diffusions via Quantized Q-learning for the Time-Discretized Process}
\label{disc_alg}
\begin{algorithmic}[1]
\STATE Choose a sampling interval $h > 0$.
\STATE Choose finite subsets $\mathds{X}_h \subset \mathds{X}$ and $\mathds{U}_h \subset \mathds{U}$.
\STATE Define a discretization mapping $\phi_\mathds{X}:\mathds{X} \to \mathds{X}_h$ (e.g., a nearest neighbour map).
\STATE Select a $\mathds{U}_h$-valued piecewise constant exploration process $\hat{u}(t)$.
\STATE Observe the process at discrete time steps and discretize it i.e. $\hat{X}_n:=\phi_{\mathds{X}}\left(X(h\times n)\right)$. For all $(\hat{x}, \hat{u}) \in \mathds{X}_h \times \mathds{U}_h$ update the Q values using the cost realization $c(X(k \times h), \hat{u})$:
    \begin{align}\label{Q_discounted}
    Q_{k+1}(\hat{x}, \hat{u}) = & (1 - \alpha_k(\hat{x}, \hat{u})) Q_k(\hat{x}, \hat{u}) \nonumber \\
    & + \alpha_k(\hat{x}, \hat{u}) \left( c(X(k \times h), \hat{u}) \cdot h + \beta_h \min_{v \in \mathds{U}_h} Q_k(\hat{X}_{k+1}, v) \right),
    \end{align}
    where $\beta_h = e^{-\beta \cdot h}$, and $\hat{X}_{k+1}$ is the sampled state observed after $\hat{X}_k = \hat{x}$.

\STATE We show that under suitable assumptions, the iterations $Q_k : \mathds{X}_h \times \mathds{U}_h \to \mathds{R}$ converge almost surely to some $Q^* : \mathds{X}_h \times \mathds{U}_h \to \mathds{R}$.  We will show that the limit values are the $Q$-values of a finite controlled Markov chain. Define the policy $\gamma_h : \mathds{X}_h \to \mathds{U}_h$ by
\[
\gamma_h(\hat{x}) = \arg\min_{\hat{u} \in \mathds{U}_h} Q^*(\hat{x}, \hat{u}).
\]
\STATE Define the control process $U_h(t)$ as
\[
U_h(t) = \gamma_h\left( \phi_\mathds{X}(X(i \cdot h)) \right), \quad \text{for } t \in [i \cdot h, (i+1) \cdot h),
\]
i.e., $U_h$ is a piecewise constant process changing value at sampling instances according to the learned map $\gamma_h$.
\end{algorithmic}
\end{algorithm}


\section{Planning I: Discretizing Time and Near Optimality of Discrete-Time Approximate Solutions}\label{disc_time_model}

\subsection*{Approximating Controlled Markov Chains:}
Let $\Rd$ be the state space, and $\Act$ be the control action space of the controlled Markov chain. 

We define the transition probabilities as follows: For any $k\in\ZZ_{+}$,  distribution of the state $X_k^h$ conditioned on the past state and action variables, is determined by the diffusion process (\ref{E1.1}), such that, conditioned on $(X_{k-1}^h,\dots,X_0^h,U^h_{k-1},\dots,U_0^h)$, $X_k^h$ has the same distribution as 
\begin{align}\label{sampled_MDP}
X_{kh} &= X_{k-1}^h +\int_{(k-1)h}^{kh}b(X_s,U^h(s))ds \nonumber\\
& \quad +\int_{(k-1)h}^{kh}\upsigma(X_s)dW_s
\end{align} 
where $U^h(s)=U_{k-1}^h$ for all $(k-1)h \leq s < kh$ (that is the control $U^h$ is piecewise constant in time). Hence, for any $A\in\cB(\Rd)$
\begin{align*}
Pr(X_k^h\in A|X^h_{[0,k-1]},U^h_{[0,k-1]})=\mathcal{T}_h(A|X^h_{k-1},U^h_{k-1})
\end{align*}
where ${X^h}_{[0,k-1]},{U^h}_{[0,k-1]}:=X_0^h,\dots,X^h_{k-1},U^h_0,\dots,U^h_{k-1}$, such that 
\begin{align}\label{disc_time_kernel}
X_k^h \sim \mathcal{T}_h(dx_k|X^h_{k-1},U^h_{k-1})
\end{align}
where $X_k^h$ determined by (\ref{sampled_MDP}) and where  $\mathcal{T}_h$ is the transition kernel of the Markov chain which is a stochastic kernel from $\Rd\times
\Act$ to $\Rd$.

For the discrete-time model, an {\em admissible policy} is a sequence of control functions $\{\gamma_k,\, k\in \ZZ_{+}\}$ such that $\gamma_k$ is measurable with respect to the $\sigma$-algebra generated by the information variables
$
I_k^h=\{X_{[0,k]}^h,U_{[0,k-1]}^h\},\,\, k \in \NN,\,\, I_0^h=\{X_0^h\},
$ that is
\begin{equation}
\label{eq_control}
U_k^h=\gamma_k(I_k^h),\quad k\in \ZZ_{+}.\nonumber
\end{equation}
 We define $\Gamma$ to be the set of all such admissible policies. 

We also define a stage-wise cost function $c_h:\Rd\times\Act\to \RR_+$ such that for any $(x,\zeta)\in\Rd\times\Act$
\begin{align*}
c_h(x,\zeta):=c(x,\zeta)\times h
\end{align*} where $c$ is the cost function of the diffusion model\,. We are interested in the following cost evaluation criteria.\\
\textbf{Discrete-time Discounted Cost:} For each $\gamma\in \Gamma$, the associated discounted cost of the approximating discrete-time model is given by
\begin{equation}\label{E2.1APM}
\cJ_{\alpha,h}(x, c,\gamma) \,\df\, \Exp_x^{\gamma} \left[\sum_{k=0}^{\infty} \beta^k c_h(X_k^h, U_k^h) | X_{0}^h = x \right]\,,
\end{equation} for $x\in\Rd$, where $\beta \df e^{-\alpha h}$\,. The optimal cost is defined as
\begin{equation}\label{E2.1APM1}
\cJ_{\alpha,h}^{*}(x, c) \,\df\, \inf_{\gamma\in \Gamma}\cJ_{\alpha,h}(x, c,\gamma)\,.
\end{equation}
\textbf{Discrete-time Ergodic Cost:} For each $\gamma\in\Gamma$, the associated infinite horizon average cost function is defined as
\begin{align}\label{MDP1_cost}
\sE_h(x,\gamma):=\limsup_{N\to\infty}\frac{1}{Nh}\Exp_{x}^{\gamma}\left[\sum_{k=0}^{N-1} c_h(X_k^h,U_k^h)\right]
\end{align}
The optimal ergodic cost function is defined as
\begin{align}\label{MDP1_optcost}
\sE_h^*(x):= \inf_{\gamma\in\Gamma} \sE_h(x,\gamma)\,.
\end{align}
\begin{remark}\label{R1}
An important property of the MDP model we will make use of is the following one: suppose that we are given an admissible policy $\gamma\in\Gamma$ defined for the MDP, we define the following continuous-time interpolated control process $U^{h}(\cdot)$ defined as
\begin{align}\label{pccontrol}
U^{h}(t)= \gamma(X^h_{k}) \text{ for } t\in[kh, (k+1)h)
\end{align}
which is a piecewise constant control process. Then, the controlled Markov chain state process $X_k^h$ under the policy $\gamma$ and the controlled diffusion process $\tilde{X}^h_t$ under the control process defined in (\ref{pccontrol}) have the same distributions at the sampling instances if they start from the same initial points, that is, for any $k\in \ZZ_{+}$, $$X_k^h\sim \tilde{X}^h_{kh}\,.$$
\end{remark}
For any $\gamma\in \Gamma$ the interpolated continuous-time policy $U^h(\cdot)$ is defined as in (\ref{pccontrol}). The set of all interpolated continuous-time admissible policies is denoted by $\bar{\Uadm}^h$ and the set of all interpolated continuous-time stationary Markov strategies is denoted by $\bar{\Uadm}_{\mathsf{sm}}^h$\,. Also, for any discrete-time Markov chain $\{X^h_k\}$, the associated continuous-time interpolated process $X^{h}(\cdot)$ is given by
\begin{align}\label{pcState}
X^{h}(t)= X^h_k \text{ for } t\in[kh, (k+1)h)\,.
\end{align} Thus, the sample paths of the continuous-time interpolated process $X^{h}(\cdot)$ are right continuous with left limits. This leads us to consider the function space 
\begin{align*}
\sD(\Rd; 0, \infty)& \df \{\omega:[0, \infty)\to \Rd \mid \omega(\cdot) \,\,\text{is right} \nonumber\\
&\text{continuous and have left limits at every}\,\, t > 0\}\,.
\end{align*} We will consider the space $\sD(\Rd; 0, \infty)$ endowed with the Skorokhod topology (for details see \cite[Section~16, p. 166]{PBill-book})\,.

Let $\cC_{BL}(\Rd)$ be the space of bounded Lipschitz continuous functions on $\Rd$, i.e., 
\begin{align}\label{BoundLips}
\cC_{BL}(\Rd) & \df \{f\in \cC(\Rd): f \,\,\text{is Lipschitz continuous and}\nonumber\\
&\,\,\norm{f}_{BL}\df\sup_{x\neq y}\frac{|f(x) - f(y)|}{|x-y|}  + \sup_{x\in \Rd} |f(x)| < \infty\}\,. 
\end{align}

In order to utilize the weak convergence technique, we introduce the relaxed control representation of the control policies (for more details see \cite[Section~9.5]{KD92})\,. 
Let $\fB(\Act\times [0, \infty))$ be the $\sigma$-algebra of Borel subsets of $\Act\times [0, \infty)$\,. Then for any Borel measure $\hat{m}$ on $\fB(\Act\times [0, \infty))$, satisfying $\hat{m}(\Act\times [0, t]) = t$ for all $t\geq 0$, it is easy to see that there exists a measure $\hat{m}_{t}$ on $\fB(\Act)$ such that $\hat{m}(d\zeta, d t) = \hat{m}_t(d\zeta) d t$\,. Let 
\begin{align*}
&\fM(\infty) \df \{\hat{m}: \hat{m}\,\,\text{is a Borel measure on}\,\,\fB(\Act\times [0, \infty))\nonumber\\
&\,\,\text{satisfying}\,\, \hat{m}(\Act\times [0,t]) = t, \text{for any}\,\, t\geq  0,\,\, \text{with}\,\, \hat{m}_t\in \cP(\Act)\}\,
\end{align*} The space $\fM(\infty)$ can be metrized by using the Prokhorov metric over $\cP(\Act\times [0, n])$ for $n\in\NN$\,. A sequence $\hat{m}^{n}$ converges to $\hat{m}$ in $\fM(\infty)$, if the normalized restriction of $\hat{m}^{n}$ converges weakly to the normalized restriction $\hat{m}$ on $\cP(\Act\times [0, n])$ for each $n\in \NN$\,. In particular, a sequence $\hat{m}^{n}$ converges to $\hat{m}$ in $\fM(\infty)$ if for any $\phi\in \cC_c(\Act\times [0, \infty))$ we have (this is known as compact weak topology)
$$\int_{\Act\times [0, \infty)} \phi(\zeta, t)\hat{m}^{n}(d \zeta, d t)\to \int_{\Act\times [0, \infty)} \phi(\zeta, t)\hat{m}(d \zeta, d t)\,.$$

Since $\cP(\Act\times [0, n])$ is complete, separable and compact for each $n\in\NN$, the space $\fM(\infty)$ inherits those properties\,. 

Now, define
\begin{align*}
\cM(\infty) =& \bigg\{m: m\,\,\text{is a random measure taking values in}\nonumber\\
&\,\, \fM(\infty)\,\, \text{satisfying}\,\,\text{for any}\,\,0\leq s <t < \infty,\nonumber\\
&\,\,  m_{[0,s]}\,\,\text{is independent of}\,\, W_t - W_s\,\bigg\}\,.
\end{align*} Since the range space $\fM(\infty)$ is compact \footnote{ By a diagonal argument, every sequence would have a converging subsequence because every compact restriction has a convergent subsequence; thus, the space is compact because it is a metric space.}, we have that any sequence in $\cM(\infty)$ is tight\,.   

\subsection{Discounted Cost}

In this subsection, we present discrete-time approximation results for the discounted cost criterion.

\begin{theorem}\label{TD1.3}\cite[Theorem 4.4]{pradhan2025discrete}
Suppose that Assumptions \hyperlink{A1}{{(A1)}}--\hyperlink{A2}{{(A2)}} hold. Then we have
\begin{equation}
\lim_{h\to 0}\cJ_{\alpha}(x, c,U_h^*) = \cJ_{\alpha}^{*}(x, c)\quad\text{a.e.}\,\,\, x\in \Rd\,.
\end{equation}
\end{theorem}

\subsection{Ergodic Cost}

In this subsection, we present discrete-time approximation results for the ergodic cost criterion. For this cost criterion we will study the near-optimality problem under the following Lyapunov stability assumption:
\begin{itemize}
\item[\hypertarget{A4}{{(A4)}}] There exist positive constants $C_0, C_1$, a compact set $\sK$ and an inf-compact function $\Lyap\in \cC^{2}(\Rd)$ (i.e., the sub-level sets $\{\Lyap\leq k\}$ are compact or empty sets in $\Rd$\,, for each $k\in\RR$) such that for all $(x,\zeta)\in \Rd\times \Act$, we have 
\begin{align}\label{Lyap1}
\sL_{\zeta}\Lyap(x) \leq C_{0} I_{\{\sK\}}(x) -  C_1\Lyap(x)\,.
\end{align}
\end{itemize}

We note that the dissipativity condition (H2) in\cite{bayraktar2017ergodicity} implies our Assumption \hyperlink{A4}{{(A4)}}.

In view of the above Lyapunov stability assumption, combining \cite[Theorem~3.7.11]{ABG-book} and \cite[Theorem~3.7.12]{ABG-book}, we have the following complete characterization of the ergodic optimal control.
\begin{theorem}\label{TErgoOpt1}
Suppose that assumptions \hyperlink{A1}{{(A1)}}--\hyperlink{A4}{{(A4)}} hold. Then the ergodic HJB equation
\begin{equation}\label{EErgoOpt1A}
\rho = \min_{\zeta \in\Act}\left[\sL_{\zeta}V^*(x) + c(x, \zeta)\right]
\end{equation} admits a unique solution $(V^*, \rho)\in \cC^2(\Rd)\cap \sorder(\Lyap)\times \RR$ satisfying $V^*(0) = 0$\,.
Moreover, we have
\begin{itemize}
\item[(i)]$\rho = \sE^*(c)$
\item[(ii)] a stationary Markov control $v^*\in \Usm$ is an optimal control (i.e., $\sE_x(c, v^*) = \sE^*(c)$) if and only if it satisfies
\begin{equation}\label{EErgoOpt1B}
\min_{\zeta \in\Act}\left[\sL_{\zeta}V^*(x) + c(x, \zeta)\right] \,=\, \trace\bigl(a(x)\grad^2 V^*(x)\bigr) + b(x,v^*(x))\cdot \grad V^*(x) + c(x, v^*(x))\,,\quad \text{a.e.}\,\, x\in\Rd\,.
\end{equation}
\item[(iii)] for any $v^*\in \Usm$ satisfying \cref{EErgoOpt1B}, we have
\begin{equation}\label{EErgoOpt1C}
V^*(x) \,=\, \lim_{r\downarrow 0}\Exp_{x}^{v^*}\left[\int_{0}^{\uuptau_{r}} \left( c(X_t, v^*(X_t)) - \sE^*(c)\right)d t\right] \quad\text{for all}\,\,\, x\in \Rd\,.
\end{equation}
\end{itemize} 
\end{theorem}

The following theorem shows that any optimal control of the discrete-time model (\ref{MDP1_cost}) is near optimal for (\ref{ECost1}) (for $h$ small enough)\,.
\begin{theorem}\label{TNearOpt1A}\cite[Theorem 5.5]{pradhan2025discrete}
Suppose that Assumptions \hyperlink{A1}{{(A1)}}--\hyperlink{A4}{{(A4)}} hold. Then we have 
\begin{align}\label{ENearOpt1A}
\lim_{h\to 0} \sE(x_0, U_h^{*}) = \sE^*(x_0)\,.
\end{align}    
\end{theorem}

\section{Planning II: Discretization of Space and Finite State MDP Construction}\label{finiteMDP}

\subsection{ Regularity Properties for Near Optimality of Finite Model Approximations}

In this subsection, we establish continuity properties of the discrete-time approximations constructed in the previous section. 

Now that we have established near optimality of discrete-time approximations, if we can show weak Feller regularity of the discrete-time model, building on \cite[Theorem 7]{KSYContQLearning} (which in turn follows from \cite{SaYuLi15c}\cite[Theorem 4.27]{SaLiYuSpringer}), we can obtain a finite space Markov Decision Process whose solution will be near optimal for the controlled diffusion.

\begin{theorem}\label{regDiscAppr}
 Consider the discrete-time model with dynamics (\ref{sampled_MDP}). 
 \begin{itemize}
 \item[(i)] Under Assumptions \hyperlink{A1}{{(A1)}}--\hyperlink{A2}{{(A2)}}, the map
\[ \Rd \times \mathbb{U} \ni (x,\zeta) \mapsto \Exp[g(X^h_{k+1}) | X^h_{k}=x,U^h_{k}= \zeta ] \in \mathbb{R},\]
is continuous for every bounded continuous $g: \Rd \to \mathbb{R}$. That is, the discrete-time model (\ref{sampled_MDP}) is weakly continuous.
\item[(ii)] Under Assumption \hyperlink{A3}{{(A3)}}, the cost function $c_h(x,\zeta):= c(x, \zeta) h: \Rd \times \mathbb{U}  \to \mathbb{R}$ is continuous and bounded.
\end{itemize}
\end{theorem}

\textbf{Proof.} 
The probability measure on the continuous function space valued solution $X_t$ is weakly continuous in the control and following the same steps as in the proof of \cite[Theorem 4.2]{PY25MckeanA} the weak Feller property can be established.

\subsection{Finite MDP Construction}\label{fin_model_section}
Now that we have obtained an approximate model whose solution is near optimal, we can now apply the well established approximation theory building on \cite{KSYContQLearning,SaLiYuSpringer,SaYuLi15c}.

Building on \cite{SaYuLi15c}, we construct an MDP with a finite state space by dicretizing the state space $\mathds{X}=\mathbb{R}^d$. Due to the $\sigma$-compact nature, we can write $\mathds{X} = \cup_{k=1}^{\infty} B_k$ where each $B_k$ is compact. Let the quantizer be such that the $M^{th}$ bin be the over-flow bin; that is, the first $ M-1$ bins be the quantization of a compact set and the complement be assigned to $B_M$. To this end, let us define
\begin{align}\label{unif_loss-2}
L^-:=\max_{i=1,\dots,M-1}\sup_{x,x'\in B_i}\|x-x'\|.
\end{align}
Note that since $\mathds{X}$ is $\sigma$-compact, for each $M$, one can find a partition $\{B_i\}_{i=1}^M$ of the state space $\mathds{X}$ such that $L^- \rightarrow 0$ and $\bigcup_{i=1}^{M-1} B_i \nearrow \mathds{X}$ as $M \rightarrow \infty$. Note that $B_M = \mathbb{X} \setminus (\cup_{i=1}^{M-1} B_i)$. In the following result, we assume that such a sequence of partitions is used to obtain the finite-state approximate models.  

 We choose a representative state, $\hat{x}_i\in B_i$, for each disjoint set and we assume that each $B_i$ is compact. For this setting, we denote the new finite state space by
$\mathds{X}_M:=\{\hat{x}_1,\dots,\hat{x}_M\}$. The mapping from the original state space $\mathds{X}$ to the finite set $\mathds{X}_M$ is done via
\begin{align}\label{quant_map}
\phi_{\mathds{X}}(x)=\hat{x}_i \quad \text{ if } x\in B_i.
\end{align}
Furthermore, we choose a weight measure $\pi(\cdot)\in {\mathcal P}(\mathds{X})$ on $\mathds{X}$ such that $\pi(B_i)>0$ for all $i\in\{1,\dots,M\}$. We now define normalized measures using the weight measure on each separate quantization bin $B_i$ such that  
\begin{align}\label{norm_inv}
{\pi}_{i}(A):=\frac{\pi(A)}{\pi(B_i)}, \quad \forall A\subset B_i, \quad \forall i\in \{1,\dots,M\}
\end{align}
 that is ${\pi}_{i}$ is the normalized weight measure on the set $B_i$, $\hat{x}_i$ belongs to. 

We now define the cost and transition kernels for this finite set using the normalized weight measures such that for any $\hat{x}_i,\hat{x}_j\in \mathds{X}_M$
\begin{align}\label{finite_cost}
&C_h(\hat{x}_i,\zeta)=\int_{B_i}c_h(x,\zeta){\pi}_{i}(dx)\nonumber\\
&P_h(\hat{x}_j|\hat{x}_i,\zeta)=\int_{B_i}\mathcal{T}_h(B_j|x,\zeta){\pi}_{i}(dx).
\end{align}
where $\mathcal{T}_h$ is the transition model for the MDP constructed in Section \ref{disc_time_model}, see (\ref{disc_time_kernel}).

Having defined the finite state space $\mathds{X}_M$, the cost function $C_h$ and the transition model $P_h$, we can now introduce the optimal value function for this finite model. We denote the optimal cost function, which is defined on $\mathds{X}_h$, by $\hat{J}^M_h:\mathds{X}_M\to \mathds{R}$.

\subsection{Discounted Cost}\label{sectionDiscFM}

\subsubsection{Near Optimal Approximation of the Diffusion Process by a Finite MDP}\label{apprx}
We denote the optimal policy for the finite MDP constructed in Section \ref{fin_model_section} by $\gamma_h^M$, and we denote the corresponding control process obtained using $\gamma_h^M$ by $U_h^M$.

We are interested in the term
\begin{align*}
\cJ_{\alpha}(x_0, c,U^M_h) - \cJ^*_{\alpha}(x_0, c)
\end{align*}
that is how much we lose if we apply the control obtained from the finite MDP in the original diffusion process. 


We write the following
\begin{align}
\cJ_{\alpha}(x_0, c,U^M_h) - \cJ^*_{\alpha}(x_0, c)&=\cJ_{\alpha}(x_0, c,U^M_h) -\cJ_{\alpha,h}(x_0, c,\gamma^M_h)\label{term1}\\
&\quad+\cJ_{\alpha,h}(x_0,c,\gamma^M_h)-\cJ_{\alpha,h}^*(x_0,c)\label{term2}\\
&\quad+\cJ_{\alpha,h}^*(x_0,c)-\cJ^*_{\alpha}(x_0, c).\label{term3}
\end{align}

The first term above is the error between the discrete-time diffusion cost under $U^M_h$ and the controlled Markov chain under the same control (note that $U^M_h$ is induced by $\gamma^M_h$); the second term is the error between the controlled Markov chain under $\gamma^M_h$ and the optimal cost for the controlled Markov chain (that is, the loss due to finite model approximation), and finally the last term is the loss due to discrete-time approximation in terms of the difference between the value functions.

In what follows, we will analyze each term separately.

\subsubsection{Analysis of Term (\ref{term1})}
The following result provides a bound for the difference between the value function of a diffusion process controlled with piecewise constant control processes and  the value function of an MDP controlled with an admissible policy.

\begin{proposition}\cite[Proposition 3.1]{bayraktar2022approximate}\label{time_app_thm}
Consider a policy $\gamma_h$ for the discrete-time MDP constructed in Section \ref{disc_time_model} , and the corresponding control process  $U_h$ obtained using $\gamma_h$ such that
\begin{align*}
U_h(t)=\gamma_h(x(i\times h)), \text{ for } t\in[i\times h,(i+1)\times h).
\end{align*}
Under Assumption \hyperlink{A1}{{(A1)}}, we have that for any $x_0\in\mathds{X}$
\begin{align*}
\left|\cJ_{\alpha}(x_0, c,U_h) -\cJ_{\alpha,h}(x_0,c,\gamma_h)\right|\leq \frac{K h}{1-e^{-\alpha h}}\left(h+\sqrt{\frac{2h}{\pi}}\right).
\end{align*}
for some $K<\infty$ that depends on the uniform upper-bounds of the functions $b$ and $\upsigma$. In particular, $\left|\cJ_{\alpha}(x_0, c,U_h) -\cJ_{\alpha,h}(x_0,c,\gamma_h)\right|\to 0$ as $h\to0$.
\end{proposition}

\subsubsection{Analysis of Term (\ref{term2})}
This term represents the performance loss due to space discretization for the discounted cost criteria. It can be shown to go to $0$ as $M\to\infty$ for weakly continuous transition kernels. Recall that by Theorem \ref{regDiscAppr} $\mathcal{T}_h(\cdot|x,\zeta)$ is weakly continuous in $(x,\zeta)$ and the cost function $c_h(x,\zeta)$ is continuous and bounded in $(x,\zeta)$. The following is then a direct consequence of \cite[Theorem 4.27]{SaLiYuSpringer}. Note that compactness of the state space is not required. 

\begin{theorem}
Under Assumptions \hyperlink{A1}{{(A1)}}--\hyperlink{A3}{{(A3)}}, with the discretization scheme described in Section \ref{fin_model_section}, we have that 
\begin{align*}
\cJ_{\alpha,h}(x_0,c,\gamma^M_h)-\cJ_{\alpha,h}^*(x_0, c) \to 0
\end{align*}
as $M\to 0$.
\end{theorem}

While in the above, we obtained an asymptotic convergence result (and this may be satisfactory in some applications), in the following we also provide an explicit rate of convergence, that is an error bound in terms of expected quantization loss terms and quantization parameters. Recall the cost and the transition model constructed for the finite space model in (\ref{finite_cost}). We extend these to space $\mathds{X}$ by defining
\begin{align*}
&\hat{c}_h(x,\zeta)=\int_{B_i}c_h(y,\zeta){\pi}_{i}(dy)\nonumber\\
&\hat{\mathcal{T}}_h(dx_1|x,\zeta)=\int_{B_i}\mathcal{T}_h(dx_1|y,\zeta){\pi}_{i}^*(dy).
\end{align*}
for any $x\in B_i$. Furthermore, we denote by $\hat{E}$ to denote the expectation under this kernel. We denote by $\hat{\cJ}^M_{\alpha,h}(x_0)$ the optimal value function under this kernel and the corresponding stage-wise cost function.

\noindent{\bf Error Bounds Under Total Variation Continuity}

In this part, instead of working with the uniform quantization error defined in \ref{unif_loss-2}, we introduce a state dependent loss function for the cost and transition kernel estimates. We assume that there exists functions $L_c$ and $L_{\mathcal{T}}$ such that 
\begin{align}\label{cost_kernel_loss}
&|c_h(x,\zeta) - \hat{c}_h(x,\zeta)|\leq L_c(x) ,\quad \forall \zeta\in\mathds{U},\nonumber\\
&\|\mathcal{T}_h(\cdot\mid x,\zeta)-\hat{\mathcal{T}}_h(\cdot\mid x,\zeta)\|_{TV}\leq  L_{\mathcal{T}}(x),\quad  \forall \zeta\in\mathds{U}.
\end{align} 
The following is then a direct consequence:
\begin{theorem}\label{disc_error_thm}\cite[Theorem 3]{KSYContQLearning}\cite[Theorem 2.2]{bicer2025quantizer}
Let 
\begin{align*}
L(x)=L_c(x)+ \frac{\beta\|c\|_\infty}{1-\beta}L_{\mathcal{T}}(x).
\end{align*}
Under Assumptions \hyperlink{A1}{{(A1)}}--\hyperlink{A3}{{(A3)}}, with the discretization scheme described in Section \ref{fin_model_section}, we have that 
\begin{align*}
\cJ_{\alpha,h}(x_0,c,\gamma^M_h)-\cJ_{\alpha,h}^*(x_0,c)\leq \sum_{k=0}^\infty \beta^k \sup_{\gamma\in \Gamma}\Exp^{\gamma}\left[L(X_k)\right]
\end{align*}
where $\beta=e^{-\alpha h}$, $\gamma_h^M$ represents the policy designed for the finite space and discrete-time model, and where $\Gamma$ denotes the space of all stationary policies for the discrete-time model.
\end{theorem}

The next result establishes the total variation Lipschitz continuity of the transition kernel of the discrete time model:
\begin{lemma}
 Under Assumptions \hyperlink{A1}{{(A1)}}--\hyperlink{A2}{{(A2)}}, we have the following
 \begin{align*}
\|\mathcal{T}_h(\cdot|x,\zeta)-\mathcal{T}_h(\cdot|y,\zeta)\|_{TV} \leq K_{\mathcal{T}}h^{-\frac{1}{2}}\|x-y\|
 \end{align*}
for some $K_{\mathcal{T}}<\infty$ where $K_{\mathcal{T}}$ is defined in the proof.
 \end{lemma}
 \begin{proof}
Note that under Assumptions \hyperlink{A1}{{(A1)}}--\hyperlink{A2}{{(A2)}}, the process $X_t$ admits a density $p_t(x,z)$ for $x_0=x$. We then have that 
\begin{align*}
\|\mathcal{T}_h(\cdot|x,\zeta)-\mathcal{T}_h(\cdot|y,\zeta)\|_{TV} \leq \int \left|p_h(x,z)-p_h(y,z)\right|dz.
\end{align*}
Furthermore, using standard Gaussian bounds (see e.g. \cite{friedman2008partial,sheu1991some}), we have that 
\begin{align*}
\left\|\nabla_{x_0}p_h(x,z)\right\|\leq C\sqrt{d} h^{-\frac{1}{2}}h^{-\frac{d}{2}} \exp{\left(-\lambda\frac{\|x-z\|^2}{h}\right)}
\end{align*}
for some $C\geq 1$ and some $\lambda\in(0,1]$. We can then write
\begin{align*}
\int_{\mathbb{R}^d} \left|p_h(x,z)-p_h(y,z)\right|dz&=\int_{\mathbb{R}^d} \left|\int_0^1 \langle \nabla_{x_0}p_h(y+\alpha(x-y),z), (x-y)\rangle d\alpha\right|dz\\
&\leq \int_{\mathbb{R}^d} \int_0^1\|\nabla_{x_0}p_h(y+\alpha(x-y),z)\| \times\|x-y\|d\alpha dz\\
&\leq C\sqrt{d} h^{-\frac{1}{2}}h^{-\frac{d}{2}}\|x-y\|\int_0^1 \int_{\mathbb{R}^d} \exp{\left(-\lambda \frac{\|y+\alpha(x-y)-z\|^2}{h}\right)}dzd\alpha\\
&=C\sqrt{d} h^{-\frac{1}{2}}h^{-\frac{d}{2}}\|x-y\|\int_0^1 \int_{\mathbb{R}^d} \exp{\left(-\lambda\frac{\|\zeta\|^2}{h}\right)}dud\alpha\\
&=C\sqrt{d} h^{-\frac{1}{2}}h^{-\frac{d}{2}}\|x-y\|\left(\frac{\pi h}{\lambda}\right)^{\frac{d}{2}}= C\sqrt{d} h^{-\frac{1}{2}} \left(\frac{\pi}{\lambda}\right)^{\frac{d}{2}} \|x-y\|.
\end{align*}

 \end{proof}
 We can then choose $L_c$ and $L_{\mathcal{T}}$ as follows:
 \begin{align}\label{new_loss}
 &L_c(x)=\alpha_c h\int \|x-y\|\pi_i(dy), \quad x \in B_i, \quad i=1,\dots M-1\nonumber\\
 &L_{\mathcal{T}}(x)= K_{\mathcal{T}}h^{-\frac{1}{2}}   \int \|x-y\|\pi_i(dy), \quad x \in B_i, \quad i=1,\dots M-1.
 \end{align}
 For the overflow bin, we simply choose uniform bounds:
 \begin{align*}
 L_c(x)=2\|c\|_\infty h, \quad L_{\mathcal{T}}(x)=2 \text{ for } x\in B_M
 \end{align*}
where we use the fact that the total variation norm is always bounded by 2.

The result is not yet complete, as the expected discretization loss $\Exp^{\gamma}[L(X_k)]$ is not guaranteed to be sufficiently small even under finer quantization. In what follows, we bound this term using the Lyapunov stability assumption  \hyperlink{A4}{{(A4)}}, and adding more structure to the quantization scheme.  

We choose a side length parameter $N>0$, choose an integer $k\geq 1$ and set $M:=k^d$. We define the centered $d$-dimensional hyper-cube
\begin{align*}
\mathcal{K}:=\left[-\frac{N}{2},\frac{N}{2}\right]^d\subset \mathbb{X}.
 \end{align*}

 We partition each side $\mathcal{K}$ into $k$ bins uniformly with width
 \begin{align*}
\Delta=\frac{N}{k}=\frac{N}{M^\frac{1}{d}},
 \end{align*}
 which produces $M$ hyper-cubic cells of volume $\Delta^d=\frac{N^d}{M}$. We enumerate these bins as $\{B_1,\dots,B_M\}$ and set 
 \begin{align*}
B_{M+1}:=\mathbb{X}\backslash\mathcal{K}.
 \end{align*}

 With this construction, $N$ that is, the length of one-side of the compact region and $M$ that is, the number of cells within the compact region are design choices. 

 \begin{lemma}\label{loss_bound_lemma}
 Suppose that Assumption \hyperlink{A4}{{(A4)}} holds with $\mathcal{V}(x):=\|x\|^m$ for some $m\geq 1$. Let 
 \begin{align*}
 &C_1(h,\beta) =  \alpha_c h +  \frac{\beta\|c\|_\infty}{1-\beta}K_{\mathcal{T}}h^{-\frac{1}{2}}\\
 &C_2(x_0,h,\beta) =  \left(2\|c\|_\infty h + \frac{2\beta \|c\|_\infty}{1-\beta}\right)C(x_0) 
 \end{align*}
 where $C(x_0):=\max\left(\|x_0\|^m,\frac{C_0}{C_1}\right) $ where $C_0,C_1$ are given in Assumption \hyperlink{A4}{{(A4)}}. 
We then we have that 
\begin{align*}
\Exp^{\gamma}\left[L(X_k)\right] \leq C_1(\beta,h)\frac{N^d}{M}+ C_2(x_0,h,\beta)\left(\frac{N}{2}\right)^{-m}
\end{align*}
 \end{lemma}
 \begin{remark}
     The result implies that by sending $M\to\infty$, and sending $N\to\infty$ at a slower rate, we can make $\Exp^{\gamma}\left[L(X_k)\right] $ arbitrarily small. In particular, choosing $N=k^{\frac{d}{d+m}}=M^{\frac{1}{d+m}}$, we get the following bound that only depends on the number of bins $M$:
     \begin{align*}
        \Exp^{\gamma}\left[L(X_k)\right] \leq \left(C_1(h,\beta) + 2^m C_2(h,\beta)\right)M^{\frac{-m}{d+m}}
     \end{align*}
     which goes to $0$ as $M\to\infty$. As noted earlier, for an explicit condition, please see \cite[(H2)]{bayraktar2017ergodicity}.
 \end{remark}

 \begin{proof}
We start by noting that under Assumption  \hyperlink{A4}{{(A4)}} holds with $\mathcal{V}(x):=\|x\|^m$, we can show that \cite{KushnerControlStochasticSC}\cite[Prop. 5.5.1]{van2007stochastic}
\begin{align*}
\Exp^{\gamma}[\|X_t\|^m]\leq \|x_0\|^m e^{-C_1 t} + \frac{C_0}{C_1}
\end{align*}
and in particular, for all $t$
\begin{align*}
\Exp^{\gamma}[\|X_t\|^m]\leq \max(\|x_0\|^m,\frac{C_0}{C_1})=:C(x_0). 
\end{align*}
We then write 
\begin{align*}
&\Exp^{\gamma}\left[L(X_k)\right] = \int_{\mathcal{K}}L(x)\mu_k(dx) + \int_{B_{M+1}} L(x)\mu_k(dx) 
\end{align*}
where $\mu_k$ denotes the marginal distribution of the state $X_k=X(k\times h)$. Note that each bin within $\mathcal{K}$ is bounded by $\Delta^d =\frac{N^d}{M}$. Hence, we write
\begin{align*}
 \int_{\mathcal{K}}L(x)\mu_k(dx)=\int_{\mathcal{K}}L_c(x)+ \frac{\beta\|c\|_\infty}{1-\beta}L_{\mathcal{T}}(x)\mu_k(dx)\leq \left(\alpha_c h +  \frac{\beta\|c\|_\infty}{1-\beta}K_{\mathcal{T}}h^{-\frac{1}{2}}\right) \frac{N^d}{M}
\end{align*}
We can then write
\begin{align*}
\Exp^{\gamma}\left[L(X_k)\right]& \leq  \left(\alpha_c h +  \frac{\beta\|c\|_\infty}{1-\beta}K_{\mathcal{T}}h^{-\frac{1}{2}}\right) \frac{N^d}{M} + \left(2\|c\|_\infty h + \frac{2\beta \|c\|_\infty}{1-\beta}\right)\mu_k(B_{M+1}).
\end{align*}
For the overflow bin, we have
\begin{align*}
&\mu_k(B_{M+1})=\mu_k\{x=\{x_1,\dots,x_d\}:\max_i |x_i|\geq \frac{N}{2}\}\leq \mu_k(\|x\|\geq \frac{N}{2})\leq \frac{\Exp[\|X_k\|^m]}{\left(\frac{N}{2}\right)^m}
\end{align*}
Combining these bounds, we write
\begin{align*}
\Exp^{\gamma}\left[L(X_k)\right]& \leq  \left(\alpha_c h +  \frac{\beta\|c\|_\infty}{1-\beta}K_{\mathcal{T}}h^{-\frac{1}{2}}\right) \frac{N^d}{M}  + \left(2\|c\|_\infty h + \frac{2\beta \|c\|_\infty}{1-\beta}\right)C(x_0) \left(\frac{N}{2}\right)^{-m}.
\end{align*}

\end{proof}

\begin{corollary}[To Theorem \ref{disc_error_thm} and Lemma \ref{loss_bound_lemma}]
 Suppose that Assumption \hyperlink{A4}{{(A4)}} holds with $\mathcal{V}(x):=\|x\|^m$ for some $m\geq 1$. 
With the discretization scheme described in Section \ref{fin_model_section}, we have that 
\begin{align*}
\cJ_{\alpha,h}(x_0,c,\gamma^M_h)-\cJ_{\alpha,h}^*(x_0,c)\leq  \frac{C_1(\beta,h)}{1-\beta}\frac{N^d}{M}+ \frac{C_2(x_0,h,\beta)}{1-\beta}\left(\frac{N}{2}\right)^{-m}
\end{align*}
where $C_1(\beta,h)$ and $C_2(x_0,\beta,h)$ are defined in the statement of Lemma \ref{loss_bound_lemma} and where $\beta=e^{-\alpha h}$, $\gamma_h^M$ represents the policy designed for the finite space and discrete-time model.

In particular, choosing $N=M^{\frac{1}{d+m}}$, we get the following bound:
     \begin{align*}
       \cJ_{\alpha,h}(x_0,c,\gamma^M_h)-\cJ_{\alpha,h}^*(x_0,c)\leq \left( \frac{C_1(h,\beta)}{1-\beta} + \frac{2^mC_2(h,\beta)}{1-\beta}\right) M^{\frac{-m}{d+m}}.
            \end{align*}
\end{corollary}

\subsubsection{Analysis of Term (\ref{term3})}
 Recall that (\ref{term3}) deals with the optimal value function of the sampled controlled Markov chain and the optimal value function of the controlled diffusion process.

\begin{proposition}\label{constant_policy}
Under  Assumptions \hyperlink{A1}{{(A1)}} -\hyperlink{A3}{{(A3)}}
\begin{align*}
\lim_{h\to 0}\cJ_{\alpha,h}^*(x_0,c)=\cJ^*_{\alpha}(x_0, c).
\end{align*}
\end{proposition}

\begin{proof}
We denote by 
\begin{align*}
 \cJ^h_{\alpha}(x_0, c)=\inf_{U\in \bar{\Uadm}^h}\cJ_{\alpha}(x, c,U)
 \end{align*}
the optimal value for the continuous-time control problem under piecewise constant controls.

We start with the following bound
\begin{align*}
|\cJ_{\alpha,h}^*(x_0,c)-\cJ^*_{\alpha}(x_0, c)|&\leq |\cJ_{\alpha,h}^*(x_0, c)-\cJ^h_{\alpha}(x_0, c)|+|\cJ^h_{\alpha}(x_0, c)-\cJ^*_{\alpha}(x_0, c)|
\end{align*}
Note that $ \cJ^h_{\alpha}(x_0, c)$ satisfies the following Bellman equation:
\begin{align*}
 \cJ^h_{\alpha}(x_0, c)&=\inf_{\zeta\in\mathbb{U}}\left\{ \Exp\left[\int_0^he^{-\alpha s}c(X_s,\zeta)ds\right] + e^{-\alpha h } \Exp\left[ \cJ^h_{\alpha}(X_h, c)|x_0,\zeta\right]\right\}\\
 &=\inf_{\zeta\in\mathbb{U}}\left\{ \Exp\left[\int_0^he^{-\alpha s}c(X_s,\zeta)ds\right] + e^{-\alpha h } \int \cJ^h_{\alpha}(x_h, c)\mathcal{T}_h(dx_h|x_0,\zeta)\right\}.
\end{align*}
Similarly, for the optimal value of the discrete time control problem, we have
\begin{align*}
\cJ_{\alpha,h}^*(x_0,c)&=\inf_{\zeta\in\mathbb{U}}\left\{ c(x_0,\zeta)\times h+ e^{-\alpha h } \int \cJ^*_{\alpha,h}(x_h, c)\mathcal{T}_h(dx_h|x_0,\zeta)\right\}.
\end{align*}
We then have the following upper bound:
\begin{align*}
\sup_x|\cJ^h_{\alpha}(x, c)-\cJ_{\alpha,h}^*(x,c)|&\leq \sup_{\zeta} \left|\Exp_x\left[\int_0^he^{-\alpha s}c(X_s,\zeta)ds\right] - c(x,\zeta)\times h \right| +  e^{-\alpha h }\sup_x|\cJ^h_{\alpha}(x, c)-\cJ_{\alpha,h}^*(x)|\\
&\leq Kh\left(h+\sqrt{\frac{2h}{\pi}}\right)+  e^{-\alpha h }\sup_x|\cJ^h_{\alpha}(x, c)-\cJ_{\alpha,h}^*(x,c)|
\end{align*}
where we used  Assumptions \hyperlink{A1}{{(A1)}} -\hyperlink{A3}{{(A3)}} for the last step. We then have the 
\begin{align*}
\sup_x|\cJ^h_{\alpha}(x, c)-\cJ_{\alpha,h}^*(x,c)|&\leq \frac{Kh}{1-e^{-\alpha h}}\left(h+\sqrt{\frac{2h}{\pi}}\right)
\end{align*}

Finally, for the second term we have that
 $\cJ_{\alpha}(x_0, c,U_h)\to\cJ^*_{\alpha}(x_0, c)$ as $h\to 0$ where $U_h$ is the optimal piecewise constant control. This is an instance of the analysis presented in \cite{pradhan2025discrete} (see Theorem \ref{TD1.3}); see also \cite{jakobsen2019improved}.
 
\end{proof}

In view of the supporting results above, we are now ready to state the finite model approximation theorem for the discounted cost criterion.

\begin{theorem}\label{theoremDiscNearOpt}
Let the optimal policy for the finite MDP constructed in Section \ref{finiteMDP} be $\gamma_h^M$, and denote the corresponding control process obtained using $\gamma_h^M$ by $U_h^M$. Under Assumptions \hyperlink{A1}{{(A1)}}--\hyperlink{A4}{{(A4)}}, we have that for every $\epsilon > 0$, there exists $\bar{M}$ and $\bar{h}$ such that 
\[\cJ_{\alpha}(x_0, c,U^M_h) - \cJ^*_{\alpha}(x_0, c) \leq \epsilon,\]
for $h \leq \bar{h}$ and $M \geq \bar{M}$. 

\end{theorem}

\subsection{Near Optimality of Finite State MDP Construction for Ergodic Cost}\label{sectionErgFM}

For the ergodic cost criterion, we can potentially proceed with two approaches. In the first one, we will show and then build on near optimality of a discounted cost optimal or near optimal solutions for the ergodic cost problem and utilize the analysis in the previous section. To this end, in the following, we present a key result. In the second approach, we can directly work with approximating the average cost discrete-time model. 

\subsubsection{Approximation via Near Optimality of Discounted Optimal Policies for Average Cost}
In this section we study the near optimality of the discounted optimal policies for the ergodic cost criterion. This will be a critical result in establishing our near optimality results involving reinforcement learning.

Let $\alpha_n \to 0$ be a sequence of discount factors, and for every $n$, $v_{\alpha_n}^*$ be a discounted-cost optimal policy. We have the following result.

\begin{theorem}\label{TErgoOptNear}
Suppose that Assumptions \hyperlink{A1}{{(A1)}}--\hyperlink{A4}{{(A4)}} hold. Then, for any $\epsilon > 0$ there exist $N_{\epsilon} > 0$ such that
\begin{equation}\label{ETErgoNearOpt1}
\sE_{x}(c, v_{\alpha_n}^*) \leq \sE^{*}(c) + \epsilon\quad\quad \mbox{for all}\,\,\, n \geq N_{\epsilon}\,.
\end{equation}
\end{theorem}
\begin{proof}
Since $\|c\|_{\infty} \leq M,$ it follows that $\sE^{*}(c) \leq M$\,. Also, in view of Lyapunov stability (\ref{Lyap1}), we have that all $v\in\Usm$ is stable and $\inf_{v\in\Usm}\eta_v(\sB_R) > 0$ for any $R>0$ (see \cite[Lemma~3.3.4]{ABG-book} and \cite[Lemma~3.2.4(b)]{ABG-book}). Thus from \cite[Theorem~3.7.6]{ABG-book}, we deduce that there exist constants $\widehat{C}_1, \widehat{C}_2 >0$ depending only on the radius $R$ such that for all $\alpha >0$,
\begin{equation}\label{ETErgoOptNearA}
\|V_\alpha(\cdot) - V_\alpha(0)\|_{\Sob^{2,p}(\sB_{R})} \leq \widehat{C}_1 \quad \text{and}\,\,\, \sup_{\sB_R}\alpha V_{\alpha} \leq \widehat{C}_2\,.
\end{equation}We know that for $1< p < \infty$, the space $\Sob^{2,p}(\sB_R)$ is reflexive and separable, hence, as a corollary of Banach Alaoglu theorem, we have that every bounded sequence in $\Sob^{2,p}(\sB_R)$ has a weakly convergent subsequence (see \cite[Theorem~3.18.]{HB-book}). Also, we know that for $p\geq d+1$ the space $\Sob^{2,p}(\sB_R)$ is compactly embedded in $\cC^{1, \beta}(\bar{\sB}_R)$\,, where $\beta < 1 - \frac{d}{p}$ (see \cite[Theorem~A.2.15 (2b)]{ABG-book}), which implies that every weakly convergent sequence in $\Sob^{2,p}(\sB_R)$ will converge strongly in $\cC^{1, \beta}(\bar{\sB}_R)$\,. Thus, in view of the estimate (\ref{ETErgoOptNearA}), by a standard diagonalization argument and Banach Alaoglu theorem, we can extract a subsequence denoted by $\{V_{\alpha}^{n}\}$ such that for some $V_{\alpha}^*\in \Sobl^{2,p}(\Rd)$
\begin{equation}\label{ETErgoOptNearB}
\begin{cases}
V_{\alpha_{n}}\to & V^*\quad \text{in}\quad \Sobl^{2,p}(\Rd)\quad\text{(weakly)}\\
V_{\alpha_{n}}\to & V^*\quad \text{in}\quad \cC^{1, \beta}_{loc}(\Rd) \quad\text{(strongly)}\,,
\end{cases}       
\end{equation} and $\alpha_n V_{\alpha_n}(0) \to \rho$\,. Now by the standard vanishing discount argument, as in \cite[Lemma~3.7.8]{ABG-book}, taking $\alpha\to 0$ we have that the pair $(V^*, \rho)\in \cC^2(\Rd)\cap \sorder(\Lyap)\times \RR$ satisfying $V^*(0) = 0$, is the unique solution of the ergodic optimality equation (\ref{EErgoOpt1A})\,. 

Moreover, since for each $n$, $v_{\alpha_n}^*$ is a discounted-cost optimal policy, we have (see \cite[Theorem~3.5.6]{ABG-book})
\begin{equation}\label{ETErgoOptNearC}
\sL_{v_{\alpha_n}^*}V_{\alpha_n^*}(x) + c(x, v_{\alpha_n}^*(x)) = \min_{\zeta\in \Act}\left[ \sL_{\zeta}V_{\alpha_n^*}(x) + c(x, \zeta)\right]\quad \text{a.e.}\,\,\, x\in\Rd\,.
\end{equation} Let $\bar{V}_{\alpha_n^*}(x) := V_{\alpha_n^*}(x) - V_{\alpha_n^*}(0)$\,. Thus rewriting (\ref{ETErgoOptNearC}), we obtain
\begin{align}\label{ETErgoOptNearD}
\sL_{v_{\alpha_n}^*}\bar{V}_{\alpha_n^*}(x) + c(x, v_{\alpha_n}^*(x)) &= \min_{\zeta\in \Act}\left[ \sL_{\zeta}\bar{V}_{\alpha_n^*}(x) + c(x, \zeta)\right]\nonumber\\
&= \alpha\bar{V}_{\alpha_n^*}(x) + \alpha V_{\alpha_n^*}(0)\quad \text{a.e.}\,\,\, x\in\Rd\,.
\end{align}

Since space of stationary Markov strategies $\Usm$ is compact,  along some further sub-sequence (without loss of generality denoting by the same sequence) we have $v_{n}^*\to \hat{v}^*$ in $\Usm$\,. It is easy to see that
\begin{align*}
&b(x,v_{\alpha_n}^*(x))\cdot \grad \bar{V}_{\alpha_n}(x) - b(x,\hat{v}^*(x))\cdot \grad V^*(x) \nonumber\\
&=  b(x,v_{\alpha_n}^*(x))\cdot \grad \left(\bar{V}_{\alpha_n} - V^*\right)(x)  + \left(b(x,v_{\alpha_n}^*(x)) - b(x,\hat{v}^*(x))\right)\cdot \grad V^*(x)\,.
\end{align*}
Since $\bar{V}_{\alpha_n}\to V^*$ in $\cC^{1, \beta}_{loc}(\Rd)\,,$ on any compact set $b(x,v_{\alpha_n}^*(x))\cdot \grad \left(\bar{V}_{\alpha_n} - V^*\right)(x)\to 0$ strongly and by the topology of $\Usm$, we have $\left(b(x,v_{\alpha_n}^*(x)) - b(x,\hat{v}^*(x))\right)\cdot \grad V^*(x)\to 0$ weakly. Thus, in view of the topology of $\Usm$, and since $\bar{V}_{\alpha_n}\to V^*$ in $\cC^{1, \beta}_{loc}(\Rd)\,,$ as $n\to \infty$ we obtain 
\begin{equation}\label{ETErgoOptNearE}
b(x,v_{\alpha_n}^*(x))\cdot \grad \bar{V}_{\alpha_n}(x) + c(x, v_{\alpha_n}^*(x)) \to b(x,\hat{v}^*(x))\cdot \grad V^*(x) + c(x, \hat{v}^*(x))\quad\text{weakly}\,.
\end{equation}
Now, multiplying by a test function $\phi\in \cC_{c}^{\infty}(\Rd)$, integrating on both sides of (\ref{ETErgoOptNearD}), and letting $n\to\infty$ it follows that $\hat{v}^*\in \Usm$ satisfies
\begin{align}\label{ETErgoOptNearF}
\rho &= \min_{\zeta \in\Act}\left[\sL_{\zeta}V^*(x) + c(x, \zeta)\right] \nonumber\\
&=\, \trace\bigl(a(x)\grad^2 V^*(x)\bigr) + b(x,v^*(x))\cdot \grad V^*(x) + c(x, v^*(x))\,,\quad \text{a.e.}\,\, x\in\Rd\,.
\end{align}Thus, from Theorem~\ref{TErgoOpt1}, we deduce that $\hat{v}^*\in \Usm$ is an ergodic optimal policy\,.

Now, from \cite[Theorem~3.6.]{pradhan2022near}, we have that the ergodic cost as a function of policy, i.e., the map $v\to \sE_x(c, v)$ is continuous on $\Usm$\,. Since $\sE^*(c) = \sE_x(c, \hat{v}^*)$ and $v_{\alpha_n}^* \to \hat{v}^*$ in $\Usm$, we deduce that for any $\epsilon > 0$, there exists $N_{\epsilon} > 0$ such that (\ref{ETErgoNearOpt1}) holds. This completes the proof of the theorem.  \end{proof}
The following theorem shows that any near-optimal policy for the discounted cost is also near-optimal for the ergodic cost criterion\,. This is what is needed for our analysis. 

\begin{theorem}\label{NearOptDisErgo1}
Suppose that Assumptions (A1)-(A3) hold. For any $\epsilon > 0$, let $v_{\alpha}^{\epsilon}$ be an $\epsilon$-optimal policy for the $\alpha$-discounted cost. Then, for small enough $\alpha$, we have
\begin{equation}\label{ETNearOptDisErgo1A}
\sE_{x}(c, v_{\alpha}^{\epsilon}) \leq \sE^{*}(c) + \epsilon\quad\quad \mbox{a.e}\,\,\, x \in \Rd\,.
\end{equation}  
\end{theorem}
\begin{proof}
By the triangle inequality we have
\begin{align}\label{EtriMain1}
\abs{\sE_{x}(c, v_{\alpha}^{\epsilon}) - \sE^{*}(c)} \leq \abs{\sE_{x}(c, v_{\alpha}^{\epsilon}) - \alpha\cJ_{\alpha}^{v_{\alpha}^{\epsilon}}(0, c)} + \abs{\alpha\cJ_{\alpha}^{v_{\alpha}^{\epsilon}}(0, c) - \alpha V_{\alpha}(0)} + \abs{\alpha V_{\alpha}(0) - \sE^{*}(c)}     
\end{align}
By the standard vanishing discount argument as in \cite[Theorem~3.7.11]{ABG-book}, we have $\abs{\alpha V_{\alpha}(0) - \sE^{*}(c)} \to 0$ as $\alpha \to 0$\,. Thus, for $\alpha$ small enough, it follows that
\begin{align}\label{ENearOptDisErgo1AA}
\abs{\alpha V_{\alpha}(0) - \sE^{*}(c)}  \leq \frac{\epsilon}{3}
\end{align}
Moreover, since $v_{\alpha}^{\epsilon}$ is $\epsilon$-optimal, for small $\alpha$, we obtain
\begin{align}\label{ENearOptDisErgo1B}
\abs{\alpha\cJ_{\alpha}^{v_{\alpha}^{\epsilon}}(0, c) - \alpha V_{\alpha}(0)}  \leq \frac{\epsilon}{3}
\end{align}

Next we want to show that $\sE_{x}(c, v_{\alpha}^{\epsilon}), \alpha\cJ_{\alpha}^{v_{\alpha}^{\epsilon}}(0, c)$ converge to the same limit as $\alpha \to 0$\,.


From \cite[Theorem~3.6.]{pradhan2022near}, we have that the the map $v\mapsto \sE_x(c, v)$ is continuous on $\Usm$\,. Since along some sub-sequence $v_{\alpha_n}^{\epsilon} \to \hat{v}^{\epsilon}$ in $\Usm$, we deduce that $\sE_{x}(c, v_{\alpha_n}^{\epsilon}) \to \sE_{x}(c, \hat{v}^{\epsilon})$\,. Furthermore, for each $n\in \NN$, the discounted cost $\cJ_{\alpha}^{v_{\alpha_n}^{\epsilon}}(x, c)$ associated to the policy $v_{\alpha_n}^{\epsilon}$ satisfies the following equation
\begin{align}\label{ENearOptDisErgo1C}
\sL_{v_{\alpha_n}^{\epsilon}}\bar{\cJ}_{\alpha}^{v_{\alpha_n}^{\epsilon}}(x, c) + c(x, v_{\alpha_n}^{\epsilon}(x)) = \alpha\bar{\cJ}_{\alpha}^{v_{\alpha_n}^{\epsilon}}(x, c) + \alpha \cJ_{\alpha}^{v_{\alpha_n}^{\epsilon}}(0, c)\quad \text{a.e.}\,\,\, x\in\Rd\,.
\end{align} where $\bar{\cJ}_{\alpha}^{v_{\alpha_n}^{\epsilon}}(x, c) := \cJ_{\alpha}^{v_{\alpha_n}^{\epsilon}}(x, c) - \cJ_{\alpha}^{v_{\alpha_n}^{\epsilon}}(0, c)$\,.
Since $c$ is bounded, we have $\sE_{x}(c, v_{\alpha}^{\epsilon}) \leq \|c\|_{\infty}$\,. Moreover, the Lyapunov stability (\ref{Lyap1}), implies that $\inf_{v\in\Usm}\eta_v(\sB_R) > 0$ for any $R>0$, where $\eta_v$ is the unique invariant measure of \eqref{E1.1} under $v\in \Usm$ (see \cite[Lemma~3.3.4]{ABG-book} and \cite[Lemma~3.2.4(b)]{ABG-book}). Hence from \cite[Theorem~3.7.4]{ABG-book}, we obtain
\begin{equation}\label{ENearOptDisErgo1D}
\|\cJ_{\alpha}^{v_{\alpha_n}^{\epsilon}}(x, c) - \cJ_{\alpha}^{v_{\alpha_n}^{\epsilon}}(0, c)\|_{\Sob^{2,p}(\sB_{R})} \leq \bar{C}_1(R) \quad \text{and}\,\,\, \sup_{\sB_R}\alpha \cJ_{\alpha}^{v_{\alpha_n}^{\epsilon}}(x, c) \leq \bar{C}_2(R)\,.
\end{equation} for some positive constant $\bar{C}_1(R), \bar{C}_2(R)$ (depending on $R$)\,. Thus, by the Banach Alaoglu theorem, and a standard diagonalization argument, we can extract a subsequence (denoted by the same sequence) such that for some $\hat{V}\in \Sobl^{2,p}(\Rd)$
\begin{equation}\label{ENearOptDisErgo1E}
\begin{cases}
\bar{\cJ}_{\alpha}^{v_{\alpha_n}^{\epsilon}}(x, c) \to & \hat{V}\quad \text{in}\quad \Sobl^{2,p}(\Rd)\quad\text{(weakly)}\\
\bar{\cJ}_{\alpha}^{v_{\alpha_n}^{\epsilon}}(x, c) \to & \hat{V} \quad \text{in}\quad \cC^{1, \beta}_{loc}(\Rd) \quad\text{(strongly)}\,,
\end{cases}       
\end{equation} and $\alpha_n \cJ_{\alpha}^{v_{\alpha_n}^{\epsilon}}(0, c) \to \hat{\rho}$\,.
Multiplying the both sides of the equation \eqref{ENearOptDisErgo1C} by a test function $\phi\in \cC_{c}^{\infty}(\Rd)$, integrating and letting $n\to\infty$ it follows that $(\hat{\rho}, \hat{V})\in \RR \times \Sobl^{2,p}(\Rd)$ satisfies
\begin{align}\label{ENearOptDisErgo1F}
\hat{\rho} = \trace\bigl(a(x)\grad^2 \hat{V}(x)\bigr) + b(x,\hat{v}^{\epsilon}(x))\cdot \grad \hat{V}(x) + c(x, \hat{v}^{\epsilon}(x))\,,\quad \text{a.e.}\,\, x\in\Rd\,.
\end{align}
Now, by the It\^{o}-Krylov formula, from \eqref{ENearOptDisErgo1C} we have
\begin{align*}
\cJ_{\alpha}^{v_{\alpha_n}^{\epsilon}}(x, c) = \Exp_x^{v_{\alpha_n}^{\epsilon}}\left[\int_0^{\uuptau_{r}} e^{-\alpha t}c(X_s, v_{\alpha_n}^{\epsilon}(X_s)) d{s} + e^{-\alpha \uuptau_{r}}\cJ_{\alpha}^{v_{\alpha_n}^{\epsilon}}(X_{\uuptau_{r}}, c)\right]\,.
\end{align*} Subtracting $\cJ_{\alpha}^{v_{\alpha_n}^{\epsilon}}(0, c)$ from both sides of the above inequality, we deduce the following:
\begin{align}\label{ENearOptDisErgo1H}
\bar{\cJ}_{\alpha}^{v_{\alpha_n}^{\epsilon}}(x, c) = & \Exp_x^{v_{\alpha_n}^{\epsilon}}\left[\int_0^{\uuptau_{r}} e^{-\alpha t}c(X_s, v_{\alpha_n}^{\epsilon}(X_s)) d{s} + e^{-\alpha \uuptau_{r}}\cJ_{\alpha}^{v_{\alpha_n}^{\epsilon}}(X_{\uuptau_{r}}, c) - \cJ_{\alpha}^{v_{\alpha_n}^{\epsilon}}(0, c)\right]\nonumber\\
= & \Exp_x^{v_{\alpha_n}^{\epsilon}}\left[\int_0^{\uuptau_{r}} e^{-\alpha t}\left(c(X_s, v_{\alpha_n}^{\epsilon}(X_s)) - \hat{\rho}\right) d{s}\right] + \Exp_x^{v_{\alpha_n}^{\epsilon}}\left[\cJ_{\alpha}^{v_{\alpha_n}^{\epsilon}}(X_{\uuptau_{r}}, c) - \cJ_{\alpha}^{v_{\alpha_n}^{\epsilon}}(0, c)\right]\nonumber\\
&+ \Exp_x^{v_{\alpha_n}^{\epsilon}}\left[\frac{1-e^{-\alpha \uuptau_{r}}}{\alpha}\left(\hat{\rho} - \alpha\cJ_{\alpha}^{v_{\alpha_n}^{\epsilon}}(X_{\uuptau_{r}}, c)\right)\right]\nonumber\\
\leq & \sup_{v\in \Usm} \Exp_x^{v}\left[\int_0^{\uuptau_{r}}\left(c(X_s, v(X_s)) + \hat{\rho}\right) d{s}\right] + \Exp_x^{v_{\alpha_n}^{\epsilon}}\left[\cJ_{\alpha}^{v_{\alpha_n}^{\epsilon}}(X_{\uuptau_{r}}, c) - \cJ_{\alpha}^{v_{\alpha_n}^{\epsilon}}(0, c)\right]\nonumber\\
&+ \Exp_x^{v_{\alpha_n}^{\epsilon}}\left[\frac{1-e^{-\alpha \uuptau_{r}}}{\alpha}\left(\hat{\rho} - \alpha\cJ_{\alpha}^{v_{\alpha_n}^{\epsilon}}(X_{\uuptau_{r}}, c)\right)\right]
\end{align}
Since, under the Lyapunov stability assumption each $v_{\alpha_n}^{\epsilon} \in \Usm$ is stable thus $\Exp_x^{v_{\alpha_n}^{\epsilon}}\left[\frac{1-e^{-\alpha \uuptau_{r}}}{\alpha}\right] = \Exp_x^{v_{\alpha_n}^{\epsilon}}\left[\int_{0}^{\uuptau_{r}}e^{-\alpha t}\right] \leq \Exp_x^{v_{\alpha_n}^{\epsilon}}\left[\uuptau_r\right] < \infty$. Moreover, from \eqref{ENearOptDisErgo1E}, we have $\sup_{\sB_r} \abs{\hat{\rho} - \alpha\cJ_{\alpha}^{v_{\alpha_n}^{\epsilon}}} \to 0$ (since $\abs{\hat{\rho} - \alpha\cJ_{\alpha}^{v_{\alpha_n}^{\epsilon}}(0, c)} \to 0$ and $\bar{\cJ}_{\alpha}^{v_{\alpha_n}^{\epsilon}}$ is bounded on compact sets uniformly in $\alpha$)\,. Now letting $\alpha \to 0 $ along sub-sequence $\alpha_n$, from \eqref{ENearOptDisErgo1H} we deduce that
\begin{align*}
\hat{V}(x)\leq & \sup_{v\in \Usm} \Exp_x^{v}\left[\int_0^{\uuptau_{r}}\left(c(X_s, v(X_s)) + \hat{\rho}\right) d{s}\right] + \Exp_x^{\hat{v}{\epsilon}}\left[\hat{V}(X_{\uuptau_{r}})\right]\,.
\end{align*} This implies
\begin{align}\label{ENearOptDisErgo1I}
\abs{\hat{V}(x)}\leq & \sup_{v\in \Usm} \Exp_x^{v}\left[\int_0^{\uuptau_{r}}\left(c(X_s, v(X_s)) + \hat{\rho}\right) d{s}\right] + \sup_{x \in \sB_r}\hat{V}(x)\,. 
\end{align} Thus, from \cite[Lemma~3.7.2]{ABG-book}, we obtain $\hat{V} \in \sorder(\Lyap)$\,. Now, by It\^{o}-Krylov formula, from \eqref{ENearOptDisErgo1F}, we get

\begin{align*}
\Exp_x^{\hat{v}^{\epsilon}}\left[\hat{V}(X_{T})\right] - \hat{V}(x) &= \Exp_x^{\hat{v}^{\epsilon}}\left[\int_0^{T}\sL_{\hat{v}^{\epsilon}} \hat{V}(X_t) dt\right]\nonumber\\
&= \Exp_x^{\hat{v}^{\epsilon}}\left[\int_0^{T} \left(\hat{\rho} - c(X_t,\hat{v}^{\epsilon}(X_t))\right)dt\right]
\end{align*}

Since $\hat{V}\in \sorder{(\Lyap)}$, in view of the results in \cite[Lemma~3.7.2(ii)]{ABG-book}, dividing both sides of the above equation by $T$ and letting $T\to \infty$, we deduce that
$\hat{\rho} = \sE_{x}(c, \hat{v}^{\epsilon})$.  
Therefore, for small $\alpha$, we obtain
\begin{align}\label{ENearOptDisErgo1G}
\abs{\sE_{x}(c, v_{\alpha}^{\epsilon}) - \alpha\cJ_{\alpha}^{v_{\alpha}^{\epsilon}}(0, c)}  \leq \frac{\epsilon}{3}.
\end{align} Now combining \eqref{ENearOptDisErgo1AA}, \eqref{ENearOptDisErgo1B}, \eqref{ENearOptDisErgo1G}, from \eqref{EtriMain1} we get \eqref{ETNearOptDisErgo1A}\,. This completes the proof\,.
\end{proof}

\subsubsection{Finite Model Approximation for Average Cost via a Direct Method}

In a 2nd approach, we can directly work with a finite model approximation under the average cost criterion. However, this approach comes with limitations for the learning algorithm: While convergence of value functions can be shown, the near optimality of learned policies requires further conditions on the model in view of the analysis in \cite{ky2023qaverage}. Notably, the lack of a minorization condition for the time-discretized diffusion model prevents the analysis in \cite{ky2023qaverage} to be directly applicable for near optimality of learned policies. Please see Appendix \ref{directArgAvg}. 

\section{Learning: Quantized Q-Learning for Controlled Diffusions and Convergence to a Near Optimal Policy}\label{QDiscDiff}

\subsection{Discounted Cost Criterion} Recall Algorithm \ref{disc_alg}. By \cite[Theorem 9]{KSYContQLearning}, this algorithm converges to a limit which gives an optimal policy for an approximate finite model defined in Section \ref{fin_model_section}. The weighting measure is defined by the restriction to the bins of the invariant measure corresponding to the exploration policy used in Q-learning. 

The policy obtained from Algorithm \ref{disc_alg} is then near optimal for the diffusion process by Theorem \ref{theoremDiscNearOpt} as a result of the analysis in Section \ref{sectionDiscFM}.

\subsection{Ergodic Cost Criterion} Given the key supporting result presented in Theorem \ref{NearOptDisErgo1} on near optimality of a (near optimal) solution for the discounted cost criterion for sufficiently high discount parameters, parallel to the argument for the discounted cost setting noted above, Algorithm \ref{disc_alg} leads to a policy which is near optimal for the ergodic cost criterion.

\begin{remark}In Section \ref{directArgAvg} we present an alternative Q-learning algorithm for the average cost setup, which also converges, and its value also converges to the value of the original diffusion problem. For further discussion, please see Section \ref{directArgAvg}.
\end{remark}



\section{Simulation}

In the following, we consider two examples. The first example satisfies all of the presented technical conditions for both discounted and average cost criteria, whereas the second example does not satisfy the conditions required for the average cost criterion. Since near optimality of Euler-Maruyama solutions is established under either of the criteria \cite{pradhan2025discrete}, we work with such a simulation. 

\subsection{Double-Well Stochastic Differential Equation}

We consider a controlled double-well SDE given by
\begin{equation}
dZ_t = (Z_t - Z_t^3 + U_t) \, dt + \upsigma \, dW_t
\end{equation}
The uncontrolled deterministic system $(\zeta=0)$ has two  stable equilibria at $z = \pm 1$ and an unstable equilibrium at $z = 0$.

The control $u_t$ aims to \emph{stabilize the system in unstable equilibrium} $z = 0$.

The running cost is given by:
\begin{equation}
c(z,\zeta)= Q \, z^2 + R \, \zeta^2 
\end{equation}
where $Q=1$ and $R=0.1$.

\begin{itemize}
\item \textbf{State space:} We uniformly  discretize the state space on $z \in [-1.4, 1.4]$ with discretization rate depending on the chosen $h$. 
\item \textbf{Action space:} We discretize the action space on $\zeta \in [-0.5, 0.5]$.

\item \textbf{Control intervals:} We vary $h \in \{0.41, 0.33, 0.25, 0.18, 0.1\}$.

\end{itemize}
Figure \ref{fig:double_well_eval} shows the performance of the policies for $\alpha=0.95$ for different $h$ values and space discretization rates.

\begin{figure}[h!]
    \centering
    \includegraphics[width=0.7\textwidth]{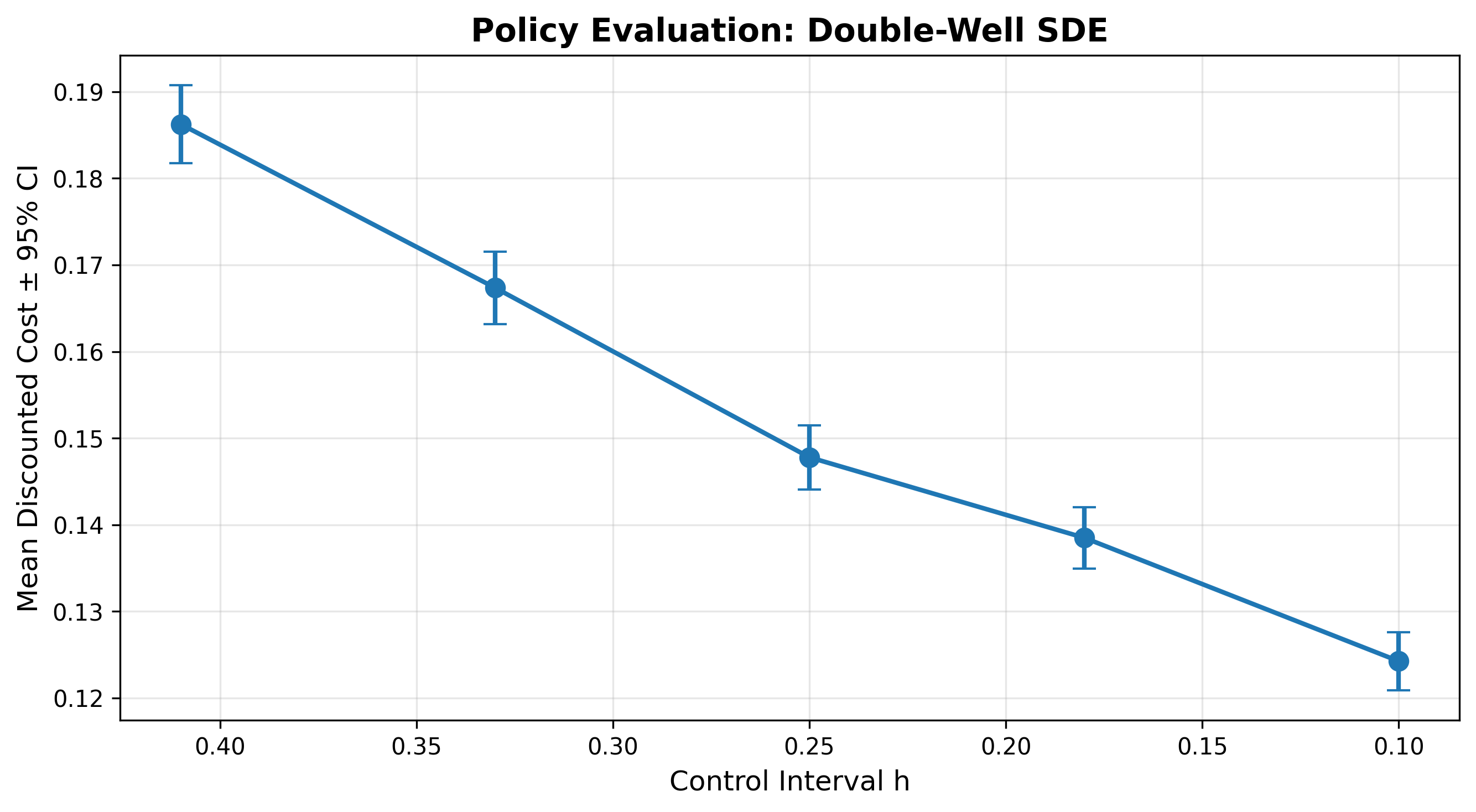}
    \caption{Average cost evaluation under policies for double-well SDE}
    \label{fig:double_well_eval}
\end{figure}

Finally, Figure \ref{vanish_dwell} shows the performance of the policies learned under different discount factors evaluated under the average cost criteria.

\begin{figure}[h!]
    \centering
    \includegraphics[width=0.6\textwidth]{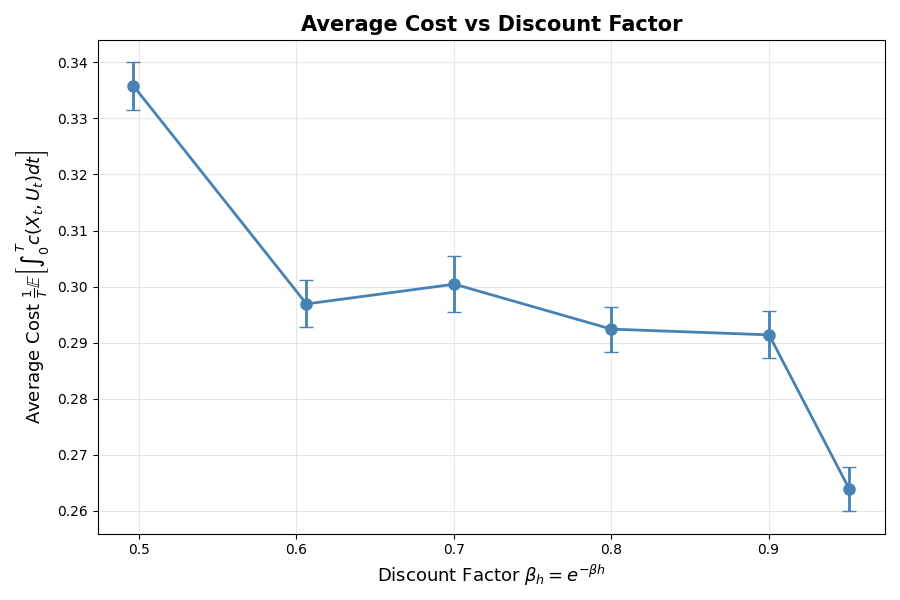}
    \caption{Average cost evaluation under different discount factors for double-well SDE}
    \label{vanish_dwell}
\end{figure}

\subsection{Logistic SDE}

For the next example, we consider a controlled logistic SDE:
\begin{equation}
    dX_t = f(X_t, U_t) \, dt + g(X_t) \, dW_t
\end{equation}
where:
\begin{itemize}
    \item $X_t \in [0, \infty)$ is the population state at time $t$
    \item $U_t \in \mathbb{R}$ is the control input (e.g., harvesting or stocking rate)
    \item The drift is $f(x, \zeta) = r x \left(1 - \frac{x}{K}\right) + \zeta$
    \item The diffusion is $g(x) = \upsigma x$
\end{itemize}

The parameters are chosen such that the growth rate $r=1$, the carrying capacity $K=1$ and the noise intensity $\upsigma=0.4$.

The objective is to minimize the infinite-horizon discounted cost:
\begin{equation}
    J_\alpha(x_0, U) = \mathbb{E}_{x_0}^{U}\left[\int_0^\infty e^{-\alpha t} c(X_t, U_t) \, dt \right]
\end{equation}
where the running cost is:
\begin{equation}
    c(x, \zeta) = Q \left(x - \frac{K}{2}\right)^2 + R \, \zeta^2
\end{equation}
with $Q= 10$, $R = 1$, and discount factor $\alpha = 0.95$.
Hence, the natural flow of the system without control is towards $K=1$, however, 
the cost penalizes deviations from the target population $K/2 = 0.5$ and large control efforts.

To apply Q-learning, we discretize both the state and action spaces:

\paragraph{State discretization:} We partition $[0, 2]$ into $n_x$ bins with grid points $\mathcal{X}_h = \{x_1, \ldots, x_{n_x}\}$.

\paragraph{Action discretization:} We discretize the control space $[-5, 5]$ into $n_u$ actions $\mathcal{U}_h = \{u_1, \ldots, u_{n_u}\}$.

\paragraph{Time discretization:} We use a control discretization parameter $h$ (the time between control updates) and a fine integration timestep $\Delta t = 0.001$ for simulating the SDE.

We change the  space discretization sizes with $h$ as shown in Table~\ref{tab:discretization}.

\begin{table}[h]
\centering
\begin{tabular}{c|c|c}
\hline
$h$ & $n_x$ (state bins) & $n_u$ (action bins) \\
\hline
0.41 & 2 & 5 \\
0.33 & 4 & 7 \\
0.25 & 6 & 9 \\
0.18 & 9 & 12 \\
0.10 & 12 & 15 \\
\hline
\end{tabular}
\caption{Number of space discretization bins for different $h$.}
\label{tab:discretization}
\end{table}

For simulating the environment, we use the Milstein scheme such that:
    \begin{equation}
        X_{k+1} = X_k + f(X_k, \zeta) \Delta t + g(X_k) \Delta W_k + \frac{1}{2} g(X_k) \upsigma (\Delta W_k^2 - \Delta t)
    \end{equation}
    where $\Delta W_k \sim \mathcal{N}(0, \Delta t)$ and where $\Delta_t=0.001$.
    
    The logistic SDE in continuous time stays positive; however, the discrete time approximation can move the approximate state to negative values. To handle this issue, we enforce $x \geq 0$ by reflecting negative values to zero.
    
     For exploration, we use pure random action selection.

Figure \ref{fig:mc_cost_h} shows the expected accumulated cost for different values of $h$. As expected the performance improves as the discretization gets finer.

\begin{figure}[h!]
    \centering
    \includegraphics[width=0.7\textwidth]{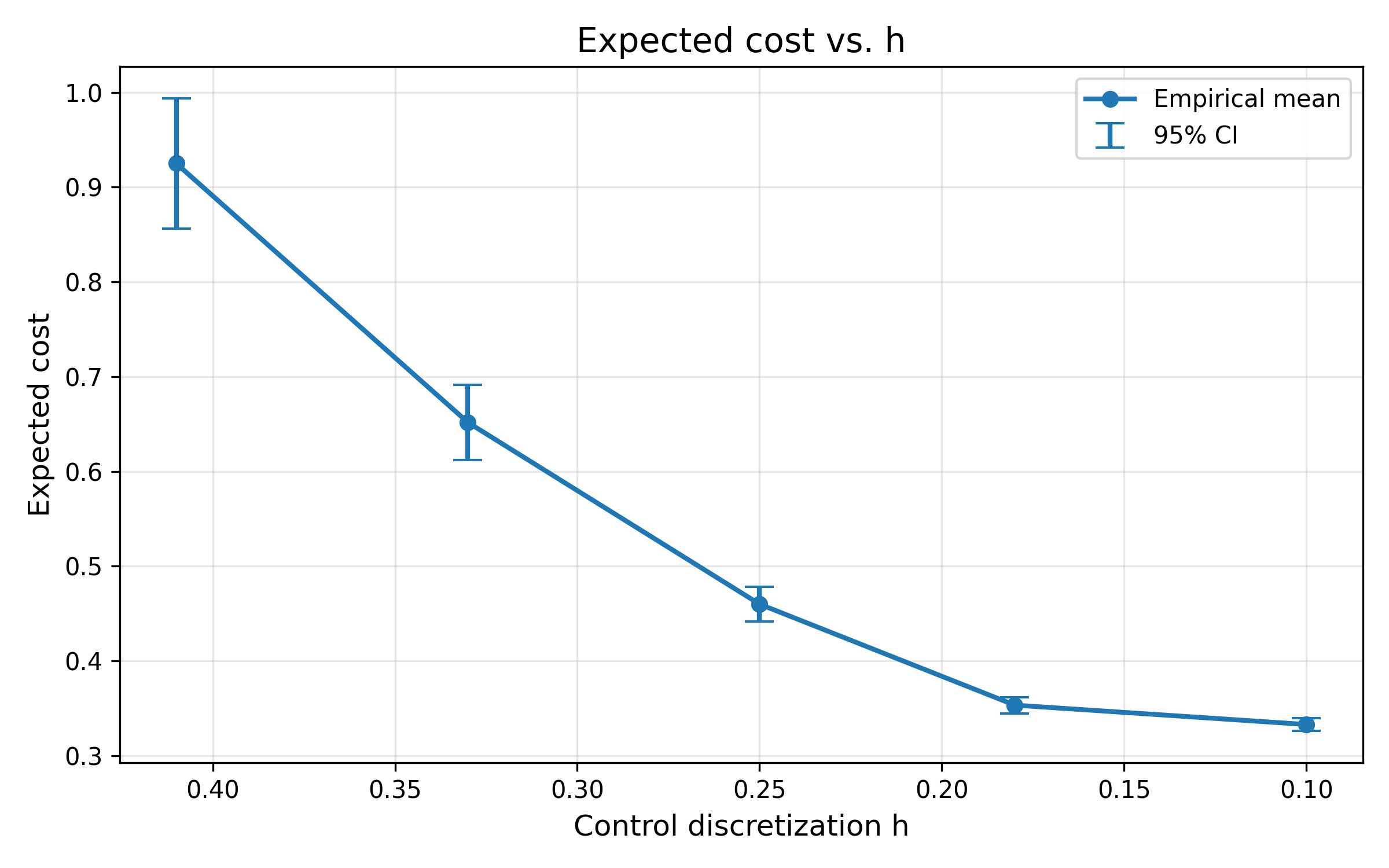}
    \caption{Expected cost versus control discretization $h$ with 95\% confidence interval for the sample mean.}
    \label{fig:mc_cost_h}
\end{figure}

We plot the average state trajectories under the learned policies in Figure \ref{fig:avg_state_paths}. As can be seen, the control does a better job of pulling the state towards $0.5$ for finer discretization rates.

\begin{figure}[h!]
    \centering
    \includegraphics[width=0.7\textwidth]{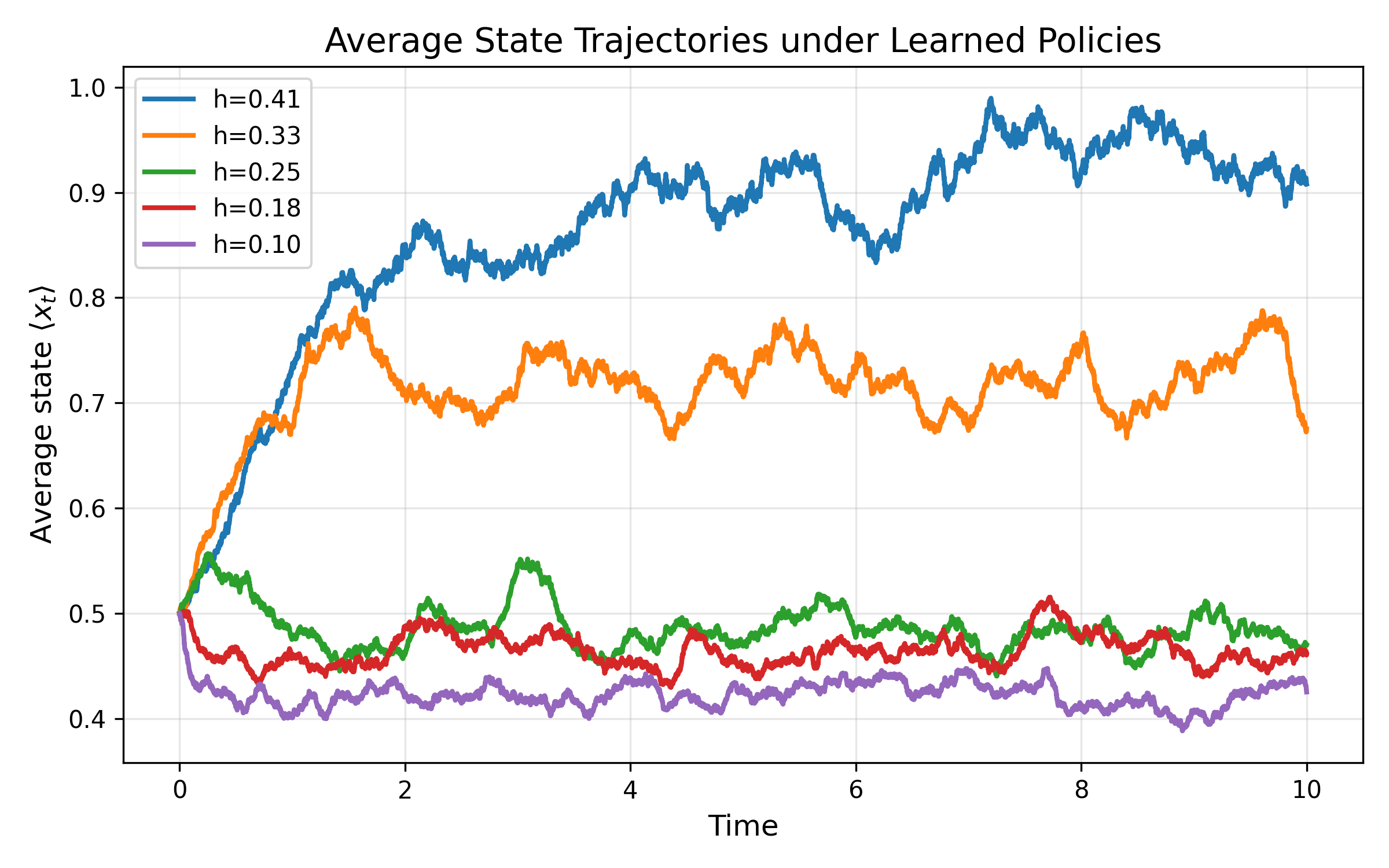}
    \caption{Average state trajectories $ x_t $ under the learned policies for different control discretizations $h$. The control pulls the state toward $K/2=0.5$, against the natural growth toward $K=1$.}
    \label{fig:avg_state_paths}
\end{figure}

Next, we run the learning algorithms using different values of discount factors $e^{-\alpha\times h}$ in increasing order that approaches to 1. While doing this, we keep $h=0.01$. In particular, we consider the values of $\alpha$ for which we have $e^{-\alpha h}=\{0.5,0.6,0.7,0.8,0.9,0.95\}$. We then evaluate the learned policies under the infinite horizon average cost criteria, i.e. for large $T$ we evaluate 
\begin{align*}
\frac{1}{T}\Exp\left[\int_{0}^{T}c(X_t,U_t) dt\right]
\end{align*}
using controls that correspond to different discount factors for $h=0.1$ and for the corresponding state discretizations that correspond to $h=0.1$. Figure \ref{fig:avg_vs_disc} shows that the cost decreases as the discount factor $e^{-\alpha h}$ approaches $1$ as expected.

\begin{figure}[h!]
    \centering
    \includegraphics[width=0.7\textwidth]{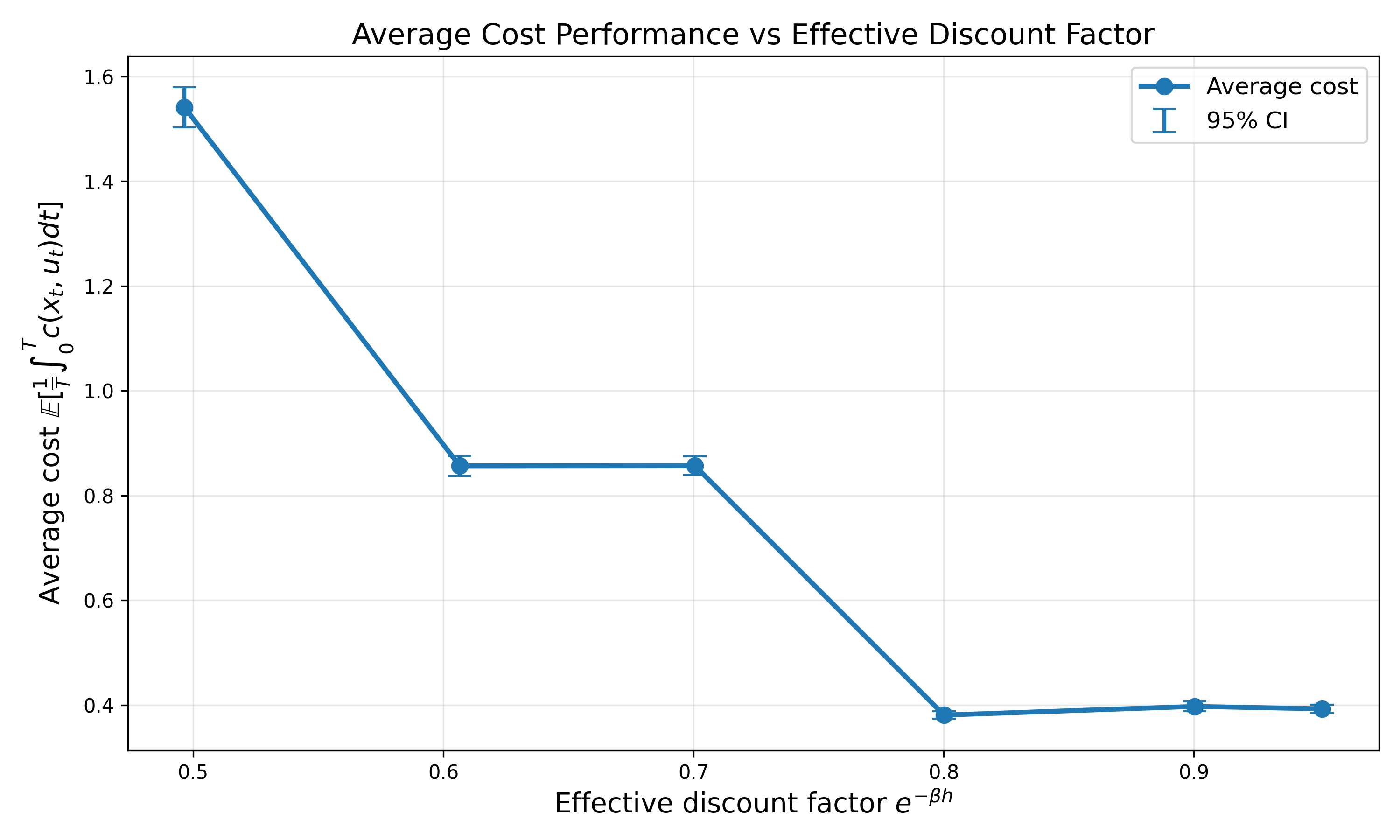}
    \caption{Average cost evaluation under policies learned under different discount factors}
    \label{fig:avg_vs_disc}
\end{figure}

\section{Conclusion}

We presented rigorous reinforcement learning results for controlled diffusions under discounted and ergodic cost criteria where the state space is $\mathbb{R}^n$ for some $n \in \mathbb{N}$. We showed that a quantized Q-learning algorithm with piecewise constant control policies applied at discrete-time instances leads to a fixed point which then leads to a near optimal control policy for each of the criteria under mild conditions. These conditions are mild in the sense that they involve only the conditions which are needed for the existence of optimal solutions.

\section{Appendix}

\subsection{A direct algorithm for average cost}\label{directArgAvg}

In this section, we present a direct algorithm for the average cost criterion. Algorithm \ref{erg_alg} below differs from Algorithm \ref{disc_alg} in the update steps (\ref{Q_discounted}) and (\ref{Q_ergodic}). 

Following \cite{ky2023qaverage}, we have that Algorithm \ref{erg_alg} presented below also converges and it also holds that the values (expected cost under an optimal policy) of finite models converge to the optimal value of the continuous model, also in view of Theorem \ref{TNearOpt1A}.  

However, while the values of the finite approximate models (as approximations get finer) converge to that of the true model, we do not yet have a proof that the obtained policy from the finite model is near optimal for the true model (this generally requires some uniformity conditions on convergence: we refer the reader to \cite{ky2023qaverage} for further discussion on such convergence properties.).
\newpage
\begin{algorithm}\caption{Q-learning Algorithm for Sampled Controlled Process for Average Cost Problems}
\label{erg_alg}
\begin{algorithmic}[1]
\STATE Choose a sampling interval $h > 0$.
\STATE Choose finite subsets $\mathds{X}_h \subset \mathds{X}$ and $\mathds{U}_h \subset \mathds{U}$.
\STATE Choose a sufficiently small normalizing constant $\delta>0$
\STATE Define a discretization mapping $\phi_\mathds{X}:\mathds{X} \to \mathds{X}_h$ (e.g., a nearest neighbour map).
\STATE Select a $\mathds{U}_h$-valued piecewise constant exploration process $\hat{u}(t)$.
\STATE Observe the process at discrete time steps and discretize it i.e. $\hat{X}_n:=\phi_{\mathds{X}}\left(X(h\times n)\right)$. For all $(\hat{x}, \hat{u}) \in \mathds{X}_h \times \mathds{U}_h$ update the Q values using the cost realization $c(X(k \times h), \hat{u})$:
    \begin{align}\label{Q_ergodic}
    Q_{k+1}(\hat{x}, \hat{u}) = & (1 - \alpha_k(\hat{x}, \hat{u})) Q_k(\hat{x}, \hat{u}) \nonumber \\
    & + \alpha_k(\hat{x}, \hat{u}) \left( c(X(k \times h), \hat{u}) \cdot h +  \min_{v \in \mathds{U}_h} Q_k(\hat{X}_{k+1}, v)  - \delta \sum_{y\in \mathds{X}_h} V_k(y)\right),
    \end{align}
    where  $\hat{X}_{k+1}$ is the sampled state observed after $\hat{X}_k = \hat{x}$ and where $V_k(\hat{x}) = \min_{v\in\mathds{U}_h}Q_k(\hat{x},\hat{u})$.

\STATE The iterations $Q_k : \mathds{X}_h \times \mathds{U}_h \to \mathds{R}$ converge almost surely to some $Q^* : \mathds{X}_h \times \mathds{U}_h \to \mathds{R}$. We will also show that the limit values are the relative  $Q$-values of a finite controlled Markov chain under ergodic cost criteria. Define the policy $\gamma_h : \mathds{X}_h \to \mathds{U}_h$ by
\[
\gamma_h(\hat{x}) = \arg\min_{\hat{u} \in \mathds{U}_h} Q^*(\hat{x}, \hat{u}).
\]
\STATE Define the control process $u_h(t)$ as
\[
u_h(t) = \gamma_h\left( \phi_\mathds{X}(X(i \cdot h)) \right), \quad \text{for } t \in [i \cdot h, (i+1) \cdot h),
\]
i.e., $u_h$ is a piecewise constant process changing value at sampling instances according to the learned map $\gamma_h$.
\end{algorithmic}
\end{algorithm}

\bibliography{Somnath_Robustness,SerdarBibliography,Quantization}

\end{document}